\documentclass[reqno]{amsart}
\usepackage{amsfonts,amssymb}
\usepackage[latin1]{inputenc}
\usepackage{enumerate}

\newtheorem{thm}{Theorem}[section]
\newtheorem{lem}[thm]{Lemma}
\newtheorem{prop}[thm]{Proposition}

\theoremstyle{definition}

\newtheorem{remq}{Remark}[section]

\newcommand{\R}{\mathbb{R}}

\newcommand{\N}{\mathbb{N}}

\newcommand{\ol}{\overline}
\newcommand{\lra}{\longrightarrow}
\newcommand{\be}{\begin{enumerate}}
\newcommand{\ee}{\end{enumerate}}

\newcommand{\mC}{\mathcal{C}}
\newcommand{\mF}{\mathcal{F}}
\newcommand{\mM}{\mathcal{M}}

\newcommand{\mS}{\mathcal{S}}
\newcommand{\mU}{\mathcal{U}}
\newcommand{\mV}{\mathcal{V}}
\newcommand{\mW}{\mathcal{W}}
\newcommand{\ds}{\displaystyle}

\newcommand{\ps}[2]{\langle\,#1,#2\,\rangle}
\newcommand{\di}[1]{\lfloor #1 \rfloor}

\newcommand{\eps}{\varepsilon}
\newcommand{\llim}{\lim\limits}
\newcommand{\infl}{\inf\limits}
\newcommand{\supl}{\sup\limits}
\newcommand{\suml}{\sum\limits}

\newcommand{\m}{\mathfrak{m}}

\begin{document}

\title[Entire solutions to nonlinear scalar field equations]{Entire solutions to nonlinear scalar field equations with indefinite linear part}
\author{Gilles \'Ev\'equoz}
\author{Tobias Weth}
\address{\ Institut f\"ur Mathematik, Johann Wolfgang Goethe-Universit\"at,
Robert-Mayer-Str. 10, 60054 Frankfurt am Main, Germany}
\email{evequoz@math.uni-frankfurt.de}
\email{weth@math.uni-frankfurt.de}


\begin{abstract}
We consider the stationary semilinear Schr\"odinger equation
\begin{equation*}
 -\Delta u + a(x) u = f(x,u), \qquad u\in H^1(\R^N),
\end{equation*}
where $a$ and $f$ are continuous functions converging to
some limits $a_\infty>0$ and $f_\infty=f_\infty(u)$ as $|x|\to\infty$. In the indefinite
setting where the Schr{\"o}dinger operator $-\Delta 
+a$ has negative eigenvalues, we combine a reduction method with a topological argument to prove the existence 
of a solution of our problem under weak one-sided
asymptotic estimates. The minimal energy level need not be attained in this case.
In a second part of the paper, we prove the existence of ground-state
solutions under more restrictive assumptions on $a$ and $f$. We stress
that for some of our results we also allow zero to lie in the spectrum of 
$-\Delta + a$.
\end{abstract}
\maketitle

\section{Introduction and main results}
Consider the semilinear elliptic equation
\begin{equation}\label{eqn:nls_f}
  -\Delta u + a(x) u = f(x,u), \qquad u\in H^1(\R^N)
\end{equation}
where $a$ and $f$ are continuous functions, $f$ being superlinear and subcritical. 
We are interested in the existence of nontrivial 
solutions in the case where 
\begin{itemize}
  \item[(A1)] $\llim_{|x|\to\infty}a(x)=a_\infty$ and $\llim_{|x|\to\infty}f(x,u)=f_\infty(u)$ hold
  uniformly for $u$ in \\ bounded sets, for some $a_\infty>0$ and $f_\infty\in\mC(\R)$.
\end{itemize}

In the definite case, $\infl_{x\in\R^N} a(x)>0$, the existence of solutions to Problem \eqref{eqn:nls_f} has
been extensively studied over the past twenty-five years, see
e.g. \cite{BaLions97, cds, CW04, dn, h, li, zh}, and most attention has been
given to nonlinearities of the type 
\begin{equation} \label{eq:6}
  f(x,u)=q(x)|u|^{p-2}u   
\end{equation}
with $p>2$, $p <\frac{2N}{N-2}$ in case $N \ge 3$ and a positive
function $q$ on $\R^N$ converging to some positive limit $q_\infty$ as
$|x| \to \infty$. The main issue in studying
this equation is to overcome the lack of compactness of the problem. 
For example, the associated energy functional
\[ J(u)=\frac{1}{2}\int_{\R^N}|\nabla u|^2+a(x)u^2\, dx -\int_{\R^N}F(x,u)\, dx,\]
where $\ds F(x,u)=\int_0^u f(x,s)\, ds$, does not satisfy the Palais-Smale condition, since the 
embedding $H^1(\R^N)\hookrightarrow L^p(\R^N)$ is not compact. Furthermore, the set of solutions of the limit problem
\begin{equation}\label{eqn:nls_f_inf}
  \left\{ \begin{array}{c} -\Delta u + a_\infty u = f_\infty(u), \\ u\in H^1(\R^N),\end{array}\right.
\end{equation}
is invariant under translations and hence not compact. On the other
hand, the concen\-tration-compactness principle
of P.-L. Lions \cite{LIONS84_1, LIONS84_2} provides a tool to
understand the nature of the lack of compactness. Using this
principle, Ding and Ni \cite{dn} established the existence of a
ground-state solution in the special case \eqref{eq:6}, $a\equiv 1$
and assuming $q_\infty = \inf \limits_{\R^N}q$. Here, by a ground-state
solution, we mean a solution with least possible energy value. 
It is easy to see that such a solution, which can be obtained by
constrained minimization, does not exist in the special case where
$a\equiv 1$, \eqref{eq:6} holds and 
$q_\infty > q(x)$ for all $x \in \R^N$. 
On the other hand, assuming \eqref{eq:6}, $a\equiv 1$ and only the
weak one-sided estimate
\[q(x)\geq q_\infty - C\, e^{-(2+\delta)\sqrt{a_\infty}|x|},\]
for some $C, \delta>0$, Bahri and Li \cite{BaLi90} (see also
\cite{BaLions97}) still could prove the existence of a positive
solution of \eqref{eqn:nls_f} by topological arguments combined with a
minimax principle. This solution is -- in general -- not a ground-state solution. 

The main purpose of the present article is to extend the two kinds of
results mentioned above to the (possibly) indefinite case, i.e., to the case
where $\inf\sigma(-\Delta+a)<0$. Here and in the following, $\sigma(-\Delta +a)$ denotes the
spectrum of the operator $-\Delta + a$. Since it follows from (A1)
that the essential spectrum of $-\Delta +a$ is given as the interval
$[a_\infty,\infty)$, the nonpositive part of
$\sigma(-\Delta +a)$ may only consist of finitely many isolated
eigenvalues. In particular,
the operator $-\Delta +a$ is negative (semi-)definite on a
finite-dimensional subspace. In order to obtain results in this
setting, one has to control the effect of this negative spectral
subspace. Our approach to this problem uses the generalized Nehari
manifold $\mM$ corresponding to \eqref{eqn:nls_f}, which -- in a
different setting -- was
introduced by Pankov \cite{PANKOV05} and studied further in
\cite{SW2009,W} (see also \cite{RT,RY} for a related
approach). The set $\mM$ -- which will be defined in
Section~\ref{sec:preliminaries} below -- contains all solutions of
\eqref{eqn:nls_f}, and minimizers of $J$ on $\mM$ are solutions
of \eqref{eqn:nls_f}. Therefore it is natural to call these minimizers
ground-state solutions of \eqref{eqn:nls_f}. As in the definite case,
one may therefore distinguish between ground-state solutions and
further solutions obtained, e.g., by minimax principles on $\mM$ relying
on topological arguments. We note that 
some existence results in the indefinite case have already been
obtained by Huang and Wang \cite{HuangWang08} using a classical linking
theorem instead of the generalized Nehari manifold. We will show that, under weaker assumptions than 
in \cite{HuangWang08}, a ground-state solution of
\eqref{eqn:nls_f} exists. Moreover, we will also treat asymptotic
conditions on $a$ and $f$ where -- similarly as in the paper \cite{BaLi90} for the
definite case -- no ground-state solution can be expected to exist.

Another aim of this paper is to allow for more general nonlinearities $f$ than in previous
papers. In order to state our first
main result, we list assumptions on $f \in\mC(\R^N\times\R)$. 
\begin{itemize}
  \item[(F1)] $|f(x,u)|\leq C_0(1+|u|^{p-1})$ for all $(x,u)\in\R^N\times\R$ with some constant $C_0>0$
  and some $2<p<2^\ast=\frac{2N}{N-2}$ if $N\geq 3$, resp. $2<p<\infty$ if $N=1,2$;

  \item[(F2)] $f(x,u)=o(|u|)$  as $|u|\to 0$, uniformly in $x$;
 
  \item[(F3)] $\ds\frac{F(x,u)}{|u|^2}\to\infty$, uniformly in $x$, as $|u|\to\infty$, where $\ds F(x,u)=\int_0^u f(x,s)\, ds$;
 
  \item[(F4)] The mappings $\ds u\mapsto \frac{f(x,u)}{|u|}$ and $\ds u\mapsto \frac{f_\infty(u)}{|u|}$ are strictly increasing in
  $(-\infty,0)\cup(0,\infty)$ for all $x\in\R^N$;

  \item[(F5)] $f_\infty$ is odd, and for some $\theta>0$ the mapping $\ds u\mapsto \frac{f_\infty(u)}{u^{1+\theta}}$
  is decreasing on $(0,\infty)$,
\end{itemize}
We also need the following stronger variant of (F2).
\begin{itemize}
  \item[(F2$'$)] There exists $\nu>0$ such that $f(x,u)=o(|u|^{1+\nu})$  as $|u|\to 0$, uniformly in $x$;
\end{itemize}

Our first result reads as follows.
\begin{thm}\label{thm:exist}
  Suppose (A1), (F1), (F2$\,'\!$) and (F3)--(F5) hold and $0 \notin \sigma(-\Delta +a)$. If
  $N\geq 2$ and the limit problem \eqref{eqn:nls_f_inf} admits a unique positive
  solution (up to translations), then \eqref{eqn:nls_f} has at least one nontrivial solution provided
  there exists $C_1, C_2\geq 0$ and $\alpha>2$ such that
  \begin{equation}\label{eqn:thm1.1:a_f}
    a(x)\leq a_\infty + C_1 e^{-\alpha\sqrt{a_\infty}|x|}\text{ and } F(x,u)\geq F_\infty(u)-C_2 e^{-\alpha\sqrt{a_\infty}|x|}(u^2 + u^p)
  \end{equation}  
  holds for all $x\in\R^N$, $u> 0$.
\end{thm}

We point out that nonlinearities of the type \eqref{eq:6} satisfy
(F1)--(F5), as well as the weakly growing superlinear nonlinearity $f(x,u)=q(x)u\log(1+|u|^s)$ with $s> 0$,
provided $q(x)>0$ for all $x\in\R^N$ and $\llim_{|x|\to\infty}q(x)=
q_\infty>0$ hold. The positive ground-state
associated to the limit problem is unique for these nonlinearities,
as follows from Theorem 1.1 in \cite{JANG10}. Moreover, the asymptotic estimate
\eqref{eqn:thm1.1:a_f} is fulfilled if $q(x) \ge q_\infty- C_2 e^{-\alpha \sqrt{a_\infty}|x|}$ for all $x \in \R^N$. 

For weakly growing superlinear nonlinearities of the type $f(x,u)=q(x)u\log(1+|u|^s)$, 
Theorem~\ref{thm:exist} is even new in the definite case where the function $a$ is positive in
$\R^N$. We point out that these weakly growing superlinear nonlinearities do not satisfy the usual 
{\em Ambrosetti-Rabinowitz growth condition} \cite{ambrosetti-rabinowitz:73} which guarantees the boundedness of 
Palais-Smale sequences. We also note that the assumption \eqref{eqn:thm1.1:a_f} is weaker
than the corresponding assumption in the paper \cite{BaLi90} of Bahri
and Li, where only the case $a \equiv 1$ was considered.   
In the indefinite case where $\inf \sigma(-\Delta +a)<0$, we
are not aware of any existence result under assumption \eqref{eqn:thm1.1:a_f}.
Our approach to prove Theorem~\ref{thm:exist} is strongly inspired by
the work of Bahri
and Li \cite{BaLi90} in the definite case, but there are crucial
differences. Most importantly, while the topological minimax argument
of \cite{BaLi90} is carried out on a unit sphere in a weighted
$L^p$-space, we have to use projection maps onto the generalized
Nehari manifold. Therefore the required asymptotic estimates are much
harder to derive. In particular, we need to deal with
eigenfunctions of $-\Delta +a$ corresponding to negative eigenvalues
and their asymptotic decay. In \cite{BaLi90}, these difficulties were 
avoided by assuming $a \equiv 1$ and therefore dealing with the most
simple spectral theoretic situation.

As in the definite case treated in \cite{BaLi90}, the solution 
obtained by Theorem~\ref{thm:exist} is not a ground-state
solution in general. In the following result, we show that, strengthening the condition on
$a$ or $f$ in the spirit of \cite{CW04}, the problem \eqref{eqn:nls_f} 
admits a ground-state solution even without the condition (F5) and
with (F2) instead of (F2$'$). 

\begin{thm}\label{thm:gs}
  Suppose (A1) and (F1)--(F4) hold. If there exists $\theta>0$ and $r_1>0$ such that 
  \begin{equation}\label{eqn:theta_delta}
    \inf_{\substack{x\in\R^N \\0<|u|\leq r_1}}\frac{|f(x,u)|}{|u|^{1+\theta}}>0
  \end{equation}
  holds, then \eqref{eqn:nls_f} admits a (nontrivial) ground-state solution, provided one of the following sets of conditions
  is satisfied.
  \begin{itemize}
    \item[(a)] There exists $S_0, C_1>0$ and $0<\alpha<\frac{2+\theta}{1+\theta}$ such that
    \begin{equation}\label{eqn:A2_gs_1}
      a(x)\leq  a_\infty - C_1 e^{-\alpha\sqrt{a_\infty}|x|}\quad \text{ for all }|x|\geq S_0,
    \end{equation}
    and there exists $\mu>\alpha$, $C_2\geq 0$ for which
    \begin{equation}\label{eqn:F_gs_1}
      F(x,u)\geq F_\infty(u) - C_2 e^{-\mu\sqrt{a_\infty}|x|}(u^2 +|u|^p)
      \quad\text{holds for all }x\in\R^N, \; u\in\R.
    \end{equation}
   \item[(b)] There exists $S_0>0$ for which $a(x)\leq a_\infty$ holds for all $|x|\geq S_0$, and
    for every $\eta>0$, there exists $0<\alpha<\frac{2+\theta}{1+\theta}$, $C_\eta, S_\eta>0$ such that
    \begin{equation}\label{eqn:A2_F_gs_2}
     F(x,u)\geq F_\infty(u) + C_\eta e^{-\alpha\sqrt{a_\infty}|x|},
     \quad\text{for all }|x|\geq S_\eta,\; \eta\leq |u|\leq\frac{1}{\eta}.
    \end{equation}
  \end{itemize}  
  Furthermore, if $0\notin\sigma(-\Delta+a)$ the conclusion also holds without \eqref{eqn:theta_delta},
  and every $0<\alpha<2$ is admissible in (a) and (b) above.
\end{thm}

To our knowlegde, Theorem~\ref{thm:gs} is the first result in this
(noncompact) setting yielding existence of solutions
in the case where $0$ is an eigenvalue of $-\Delta+a$. Since the eigenfunctions associated the 
eigenvalue $0$ exhibit a slower decay rate than the ones corresponding to negative eigenvalues, we cannot expect, 
in general, to allow every value $\alpha\in(0,2)$ in Theorem \ref{thm:gs}.
Nevertheless, any $\alpha\in(0,1]$ is allowed, since $\frac{2+\theta}{1+\theta}=1+\frac{1}{1+\theta}>1$ for $\theta>0$.
In the case where $0\notin\sigma(-\Delta+a)$ is considered,
Theorem~\ref{thm:gs} is a generalization of results of Huang and Wang
\cite{HuangWang08}. More precisely, we obtain
the existence of solutions under weaker assumptions upon the potential
$a$. In particular, we only need to control the behavior of
$a_\infty-a(x)$ for large $x$. Moreover, Theorem \ref{thm:gs} provides
the additional information that ground-state solutions exist. 
\medskip

The paper is organized as follows. We first state and prove some basic properties of the
energy functional and the generalized Nehari manifold. Some crucial energy estimates are
then derived before the actual proof of Theorem \ref{thm:exist} is given. In the last section
we give some further energy estimates under the assumptions of Theorem \ref{thm:gs} and conclude
by proving the latter. Finally, in the appendix we prove a nonlinear splitting property for weakly 
converging sequences in $H^1(\R^N)$, which is necessary for the decomposition of Palais-Smale
sequences of $J$. Here we adapt a result in \cite[Appendix]{AckWeth05} which
was stated for the periodic setting. In contrast to earlier results of
this type (see e.g.\cite{KrSz98}), no Lipschitz continuity of $f$ or
bounds on $f'$ are required here. 

\section{Preliminaries}\label{sec:preliminaries}
Throughout this paper we shall use the following notation. For a function $u$ on $\R^N$ and an element $y\in\R^N$,
we write $y\ast u$ for the translate of $u$, i.e., 
\[(y\ast u)(x):=u(x-y),\quad x\in\R^N.\] 
Let $X$ be any normed space, we will denote by $B_r(u)$ the open ball in $X$ centered at $u\in X$ with radius $r>0$.

According to (A1), the essential spectrum of $-\Delta + a$ is equal to $[a_\infty,+\infty)$ 
(see e.g. \cite[Theorem 3.15]{STUART98}), and $\sigma(-\Delta+a)\cap (-\infty,a_\infty)$ consists (at most)
of isolated eigenvalues of finite multiplicity. Let
\[\inf\sigma(-\Delta+a)=\lambda_1\leq \ldots\leq \lambda_n<0=\lambda_{n+1}=\ldots=\lambda_{n+l}\; (<\lambda_{n+l+1})\]
denote the nonpositive eigenvalues (repeated according to multiplicity), and consider a corresponding orthonormal
set of eigenfunctions $e_1, \ldots, e_{n+l}\in H^2(\R^N)\cap \mathcal{C}(\R^N)$.
Setting $E^-=\text{span}\{e_1, \ldots, e_n\}$, $E^0=\text{span}\{e_{n+1},\ldots,e_{n+l}\}$ (with the convention that $E^0=\{0\}$ if $l=0$)
and $E^+=(E^-\oplus E^0)^\bot$, we have the so-called spectral decomposition
\begin{equation}\label{eqn:spec_dec}
  E:=H^1(\R^N)=E^+\oplus E^0\oplus E^-
\end{equation}
corresponding to $-\Delta + a$.
Moreover, the eigenfunctions satisfy the following exponential decay
estimates (see \cite[Theorem 3.19]{STUART98}) which play a crucial role in
the sequel.

If $1\leq i\leq n$, then  
\begin{equation}\label{eqn:eigen_neg}
  \lim_{|x| \to \infty} |e_i(x)|e^{(1+\delta)\sqrt{a_\infty}|x|} = 0 \qquad \text{for every $0<\delta<
\sqrt{1+\textstyle\frac{|\lambda_i|}{a_\infty}}\,-1$.}
\end{equation}
On the other hand, if $n+1\leq i\leq n+l$, then $\lambda_i=0$ and 
\begin{equation}\label{eqn:eigen_zero}
  \lim_{|x| \to \infty} |e_i(x)|e^{(1-\delta)\sqrt{a_\infty}|x|} = 0 \qquad \text{for every $\delta>0$.}
\end{equation}

\subsection{Energy functional and generalized Nehari manifold}
We assume for the remainder of this section that (A1), (F1)--(F4) hold, and denote by $\|\cdot\|$ an equivalent norm on 
$E=H^1(\R^N)$ which satisfies
\[\int_{\R^N} |\nabla u|^2 + a(x) u^2\, dx = \|u^+\|^2 - \|u^-\|^2,\qquad\text{ for }u\in E.\]
Here and in the sequel, we let $u^\pm$ and $u^0$, respectively, be the projections of $u\in E$ onto $E^\pm$ and $E^0$, respectively,
according to the decomposition \eqref{eqn:spec_dec}.

The solutions of \eqref{eqn:nls_f} are critical points of the energy functional $J$: $E$ $\to$ $\R$ given by
\[J(u)= \frac{1}{2}(\|u^+\|^2 -\|u^-\|^2) - \int_{\R^N}F(x,u)\, dx, \quad u\in E.\]
Considering the generalized Nehari manifold (see e.g. \cite[Chapter 4]{W})
\[\mM=\{u\in E\backslash (E^-\oplus E^0)\, :\, J'(u)(tu+h)=0,\;\;t\geq 0,\;h\in E^-\oplus E^0\},\]
we set
\[c=\inf_{u\in\mM}J(u).\]
For the limit problem \eqref{eqn:nls_f_inf}, we set 
\[J_\infty(u)=\frac12\int_{\R^N} |\nabla u|^2 + a_\infty u^2\, dx - \int_{\R^N}F_\infty(u)\, dx\quad\text{ for }u\in E,\]
consider the associated Nehari manifold
$\mM_\infty=\{u\in E\backslash\{0\}\,:\, J_\infty'(u)u=0\}$ 
and let $c_\infty=\infl_{u\in\mM_\infty}J_\infty(u)$. 

We recall that, since $a_\infty>0$ holds and since $f_\infty$ satisfies the conditions (F1)--(F4), 
Problem \eqref{eqn:nls_f_inf} admits a ground-state solution $u_\infty\in E\backslash\{0\}$
(see \cite[Theorem 3.13]{W}.)
There holds $J_\infty(u_\infty)=c_\infty>0$ and $J_\infty'(u_\infty)=0$.

Before we give some properties of $\mM$, let us point out a few facts concerning the functions $f$, $f_\infty$ and their primitive
$F$, $F_\infty$.

\begin{lem}\label{lem:properties_f}
  \begin{itemize}
   \item[(i)] For every $\eps>0$ there is $C_\eps>0$ such that 
    \begin{equation}\label{eqn:estim_f_eps}
      |f(x,u)| \leq \eps |u| + C_\eps |u|^{p-1} \quad\text{ and }\quad F(x,u)\leq \eps|u|^2 + C_\eps |u|^p.
    \end{equation} 
    for all $x \in \R^N$, $u \in \R$.
   \item[(ii)] For all $x\in\R^N$ and all $u,v\in\R$, we have
    \begin{equation}\label{eqn:ineq_F_f}
     F(x,u+v)\geq F(x,u) + f(x,u)v \qquad \text{and}\qquad  F_\infty(u+v) \ge F_\infty(u)+f_\infty(u)v
    \end{equation}
   \item[(iii)] If (F2$\,'\!$) holds, then for every $\rho>0$ there exists $C_\rho\geq 0$ such that for all $0\leq u,v\leq \rho$ we have
    \begin{equation}\label{eqn:ineq_Finf_finf}
     F_\infty(u+v)-F_\infty(u)-F_\infty(v)\geq f_\infty(u)v+f_\infty(v)u - C_\rho u^{1+\frac{\nu}{2}}v^{1+\frac{\nu}{2}}.
    \end{equation}
  \end{itemize}
\end{lem}

\begin{proof}
(i) follows easily from (F1) and (F2).\\
(ii) As a consequence of (F4), the function $u \mapsto f(x,u)$ is
increasing on $\R$ for every $x \in \R^N$, which yields   
\[F(x,u+v)-F(x,u)= \int_u^{u+v}f(x,t)\,dt \ge f(x,u)v.\]
The statement on $F_\infty$ and $f_\infty$ follows in the same way
from (F4).\\
(iii) The inequality is obviously satisfied if $u=0$ or $v=0$. Moreover, for $0<v\leq u$, we deduce from \eqref{eqn:ineq_F_f} and (F2$'$)
\begin{align*}
  &F_\infty(u+v)-F_\infty(u)-F_\infty(v)-f_\infty(u)v-f_\infty(v)u \geq -F_\infty(v)-f_\infty(v)u \\
  &= -\int_0^v \frac{f_\infty(t)}{t^{1+\nu}}\, t^{1+\nu}\, dt - \frac{f_\infty(v)}{v^{1+\nu}}uv^{1+\nu}\\
  &  \geq -\frac{\tilde{C}_\rho}{(2+\nu)}v^{2+\nu}-\tilde{C}_\rho\, u v^{1+\nu} \geq -\frac32 \tilde{C}_\rho u^{1+\frac{\nu}{2}}v^{1+\frac{\nu}{2}},
\end{align*}
where $\tilde{C}_\rho:=\sup\limits_{0<u\leq \rho} \frac{f_\infty(u)}{u^{1+\nu}}<+\infty$. Since \eqref{eqn:ineq_F_f} and \eqref{eqn:ineq_Finf_finf}
are symmetric in $u$ and $v$, the same estimate holds for $0<u\leq v$, and the proof is complete.
\end{proof}

We now study more closely the set $\mM$ and the behavior of $J$ on it.
\begin{lem}[Properties of $\mM$]\label{lem:Nehari_f}
  There holds
  \begin{itemize}
    \item[(i)] $\frac{1}{2}\int_{\R^N}f(x,u)u\, dx > \int_{\R^N}F(x,u)\, dx$ for all $u\in E\backslash\{0\}$,
     and the functional $u\mapsto \int_{\R^N} F(x,u)\, dx$ is weakly lower semicontinuous.
    \item[(ii)] For each $w\in E\backslash (E^-\oplus E^0)$ let $\hat{E}(w):=\{tw+h\, :\, t\geq 0\; h\in E^0\oplus E^-\}$.
     Then there exists a unique nontrivial critical point $\hat{\m}(w)$ of $J|_{\hat{E}(w)}$.
     Moreover, $\hat{\m}(w)$ is the unique global maximum of $J|_{\hat{E}(w)}$.
    \item[(iii)] There exists $\delta>0$ such that $\|\hat{\m}(w)^+\|\geq \delta$ for all $w\in E\backslash (E^-\oplus E^0)$, and 
     for each compact subset $\mW\subset E\backslash (E^-\oplus E^0)$ there is a constant $C_{\mW}>0$ such that 
     $\|\hat{\m}(w)\|\leq C_{\mW}$ for all $w\in\mW$.
  \end{itemize}
\end{lem}
\begin{proof}
 \begin{itemize}
  \item[(i)] The first assertion follows from \cite[Lemma 2.1]{SW2009}, and  \cite[Theorem 1.6]{STRUWE} gives the second one.
  \item[(ii)] Similar to the proof of \cite[Proposition 2.3]{SW2009}, we see that for all $u\in\mM$, there holds $J(u)>J(v)$ for
  every $v\in\hat{E}(u)\backslash\{u\}$.
  
  Let now $w\in E\backslash (E^-\oplus E^0)$. It is enough to prove that $\mM\cap \hat{E}(w)\neq\varnothing$ holds. With $w^+\neq 0$,
  we set $v:=\frac{w^+}{\|w^+\|}$ and claim that for $t\geq 0$, $h\in E^-\oplus E^0$, we have $J(tv+h)\leq 0$ if $\|tv+h\|$ is large.
  Indeed, suppose, by contradiction that $\|t_kv+h_k\|\to \infty$ as $k\to\infty$ and $J(t_kv+h_k)\geq 0$ for all $k$.
  Setting $v_k:=\frac{t_kv+h_k}{\|t_kv+h_k\|}=s_kv+z_k$ for all $k$, we first note that $(s_k)_k$, $(z_k^-)_k$ and $(z_k^0)_k$ are bounded
  sequences, since $1=\|v_k\|^2 = s_k^2 + \|z_k^-\|^2 + \|z_k^0\|^2$ holds for all $k$. Thus, up to a subsequence, we can assume 
  $v_k\to s v + z$ for some $s\geq 0$ and $z\in E^-\oplus E^0$, since $\dim (E^-\oplus E^0)<+\infty$. In particular, $\|s v+z\|=1\neq 0$, 
  and therefore, $|t_kv(x)+h_k(x)|\to \infty$ as $k\to\infty$, for a.e. 
  $x\in\R^N$ such that $sv(x)+z(x)\neq 0$. Condition (F3) together with Fatou's Lemma now gives
  \[ \frac12 s_k^2 -\frac12 \|z_k^-\|^2 -\frac{J(t_kv+h_k)}{\|t_kv+h_k\|^2} = \int_{\R^N}\frac{F(x,t_kv+h_k)}{(t_kv+h_k)^2}v_k^2\, dx \to \infty,\]
  which contradicts the assumption $J(t_kv+h_k)\geq 0$ for all $k$ and thus proves the claim.
  
  Next, we notice that \eqref{eqn:estim_f_eps} implies $J(tv)>0$ for $t>0$ small. Consequently $0<\supl_{u\in\hat{E}(w)}J(u)<+\infty$,
  and we conclude as in the proof of \cite[Lemma 2.6]{SW2009}.
  \item[(iii)] A similar proof as \cite[Lemma 2.4]{SW2009} gives the first assertion. For the second one, we simply note that $\hat{\m}(w)$ has
  the form $tw^+ + h$ with $t\geq 0$ and $h\in E^-\oplus E^0$. Hence, the same argument as in the proof of (ii), together with the fact that
  $J(\hat{\m}(w))=\int_{\R^N}\frac12 f(x,\hat{\m}(w))\hat{\m}(w)-F(x,\hat{\m}(w))\, dx>0$, implies that $\hat{\m}(w)$ is uniformly bounded for
  $w\in\mW$, since this set is compact.
  \end{itemize}
\end{proof}

\begin{lem}[Coercivity]\label{lem:coercivity}
  Every sequence $(u_k)_k\subset\mM$ with $\llim_{k\to\infty}\|u_k\|=\infty$ satisfies 
  $\llim_{k\to\infty}J(u_k)=\infty$. In particular, all Palais-Smale sequences for $J$ in $\mM$ are bounded.
\end{lem}
\begin{proof}
By contradiction, let $(u_k)_k\subset\mM$ satisfy $d:=\supl_{k\in\N}J(u_k)<\infty$ and \\
$\llim_{k\to\infty}\|u_k\|=\infty$. Let $v_k:= \frac{u_k}{\|u_k\|}$
for $k \in \N$. 
We first claim that 
\begin{equation}\label{eq:1}
  \|v_k^+\| \not \to 0 \qquad \text{as $k \to \infty$.}  
\end{equation}
Indeed, suppose that $\|v_k^+\| \to 0$. Then, since 
$0<J(u_k)\le \frac12 (\|u_k^+\|^2-\|u_k^-\|^2)$, we have $\|u_k^-\| \le \|u_k^+\|$
for all $k$ and therefore 
\[\|v_k^-\| \le \|v_k^+\| \to 0 \qquad \text{as $k \to \infty$.}\]
As a consequence, since $E^0$ is finite-dimensional, we may pass to a
subsequence such that $v_k \to v$, where $v \in E^0$ satisfies $\|v\|=1$. 
Since $|u_k(x)|\to\infty$ for a.e. $x \in \R^N$ with $v(x) \ne 0$, it follows from $(F_3)$ and Fatou's lemma that
\begin{equation*}
  \int_{\R^N}\frac{F(x,u_k)}{u_k^2}v_k^2\,dx \to \infty \qquad \text{as $k \to \infty$,}
\end{equation*}
and therefore 
\[0 \le \frac{J(u_k)}{\|u_k\|^2}=\frac12 (\|v_k^+\|^2-\|v_k^-\|^2) - 
  \int_{\R^N} \frac{F(x,u_k)}{u_k^2}v_k^2\,dx \to -\infty\]
as $k \to \infty$, a contradiction. Hence \eqref{eq:1} holds, and
therefore we may pass to a
subsequence such that 
\begin{equation}\label{eq:3}
  \sigma:= \inf \limits_{k \in \N}\|v_k^+\|>0.  
\end{equation}
Next we claim that 
\begin{equation}\label{eq:2}
  v_k^+ \not \to 0 \in L^p(\R^N),  
\end{equation}
where $p >2$ is as in (F1). Indeed, suppose by contradiction that $v_k^+ \to 0$ in
$L^p(\R^N)$, and let $s>0$. Then \eqref{eqn:estim_f_eps} yields 
$\llim_{k \to \infty} \int_{\R^N}F(x,s v_k^+)\,dx=0$. Moreover, since $s v_k^+ \in \hat E(u_k)$,
Lemma~\ref{lem:Nehari_f}(ii) implies that 
\begin{align*}
d \ge J(u_k) \ge J(sv_k^+) &=  \frac12\|sv_k^+\|^2- \int_{\R^N}F(x,s v_k^+)\,dx \\
 & \ge \frac{(s\sigma)^2}{2} -  
  \int_{\R^N}F(x,s v_k^+)\,dx \to \frac{(s\sigma)^2}{2}
\end{align*}
for $k \to \infty$. Since $s>0$ was arbitrary, we get a
contradiction.

By \eqref{eq:2} and Lions' Lemma \cite[Lemma I.1]{LIONS84_2}, there exists a sequence $(y_k)_k$ in $\R^N$ such
that, after passing to a subsequence, $\inf \limits_{k \in \N}\int_{B_1(0)}(y_k\ast v_k^+)^2\, dx>0$ 
and therefore, passing again to a subsequence,
$y_k * v_k^+ \rightharpoonup v$ as $k \to \infty$, where $v \in H^1(\R^N)\setminus \{0\}$. 

Since $\text{dim}(E^-\oplus E^0)<+\infty$, we can find $z\in E^-\oplus E^0$ such that, up to a subsequence,
$ v_k^- + v_k^0 \to z$ holds, as $k\to\infty$. If the sequence $(y_k)_k$ is bounded, we even have 
$v_k^+ \rightharpoonup w$ for some $w\in E^+\backslash\{0\}$,
up to a subsequence, and therefore we obtain $v_k\rightharpoonup w+z\neq 0$ as $k\to\infty$.
Passing again to a subsequence, we may then also assume
\begin{equation}\label{eq:4}
  v_k \to  w+z \qquad \text{a.e. in $\R^N$.} 
\end{equation}
On the other hand, if $(y_k)_k$ is unbounded, we may pass to a subsequence satisfying $|y_k| \to \infty$ and, 
consequently, $y_k * (v_k^-+v_k^0) \rightharpoonup 0$ as $k \to \infty$. This gives
$y_k\ast v_k\rightharpoonup v\neq 0$, and we may pass to a subsequence
satisfying 
\begin{equation} \label{eq:5}
  y_k\ast v_k \to v \qquad \text{a.e. in $\R^N$.} 
\end{equation}
Now, we remark that 
\[\int_{\R^N}\frac{F(x,u_k)}{u_k^2}v_k^2\,dx= \int_{\R^N}\frac{F(x-y_k ,y_k * u_k)}{(y_k *u_k)^2}(y_k* v_k)^2\,dx.\]
Moreover, \eqref{eq:4} implies $|u_k|\to\infty$ pointwise
a.e. where $w+z \not = 0$, while \eqref{eq:5} implies 
$|y_k * u_k| \to \infty$ a.e. where $v \not=0$. Hence (F3) and Fatou's Lemma
again imply that 
\begin{equation*}
  \int_{\R^N}\frac{F(x,u_k)}{u_k^2}v_k^2\,dx = \int_{\R^N}\frac{F(x-y_k ,y_k * u_k)}{(y_k *u_k)^2}(y_k* v_k)^2\,dx \to \infty 
  \qquad \text{as $k \to \infty$,}
\end{equation*}
and therefore 
\[0 \le \frac{J(u_k)}{\|u_k\|^2}=\frac12 (\|v_k^+\|^2-\|v_k^-\|^2) - \int_{\R^N} \frac{F(x,u_k)}{u_k^2}v_k^2\,dx \to -\infty.\]
This contradiction finishes the proof.
\end{proof}

From Propositions 4.1, 4.2 and Corollary 4.3  in \cite{W}, it follows that
\begin{lem}\label{lem:Nehari2_f}
  \begin{itemize}
    \item[(a)] The map $\hat{\m}$: $E\backslash (E^-\oplus E^0)$ $\to$ $\mM$ given by Lemma \ref{lem:Nehari_f} (ii) is continuous 
     and its restriction $\m$ to $\mS^+:=\{u\in E^+\, :\, \|u\|=1\}$ is a homeomorphism with inverse given by 
     $\m^{-1}(v)=\frac{v^+}{\|v^+\|}$, $v\in\mM$.
    \item[(b)] The functional $\hat{\Psi}$: $E^+\backslash\{0\}$ $\to$ $\R$ defined by $\hat{\Psi}(w)=J(\hat{\m}(w))$ is of class $\mC^1$. 
     Furthermore, $\Psi:=\hat\Psi|_{\mS^+}$ is also $\mC^1$ with $\Psi'(w)z=\|\m(w)^+\| J'(\m(w))z$ for every 
     $z\in T_w\mS^+=\{v\in E^+\, :\, \ps{w}{v}=0\}$.
    \item[(c)] $(w_k)_k\subset\mS^+$ is a Palais-Smale sequence for $\Psi$ if and only if $(\m(w_k))_k\subset\mM$ is a Palais-Smale 
     sequence for $J$.
    \item[(d)] $\infl_{\mS^+}\Psi = \infl_{\mM}J =c$.
  \end{itemize}
\end{lem}

\begin{prop}\label{prop:c<cinf_f}
  Suppose (A1) and (F1)--(F4) are satisfied. 
  If in addition $c<c_\infty$ holds, then \eqref{eqn:nls_f} has a nontrivial ground-state solution.
\end{prop}
\begin{proof}
Let $(v_k)_k\subset \mS^+$ be a minimizing sequence for $\Psi$. 
By Ekeland's variational principle \cite[Theorem 8.5]{WILLEM}, we can find a sequence $(w_k)_k\subset\mS^+$ such that $\|w_k-v_k\|\to 0$, 
$\Psi(w_k)\to \infl_{\mS^+}\Psi=c$ and $\|\Psi'(w_k)\|_\ast\to 0$ as $k\to\infty$. Setting $u_k:=\m(w_k)$ for all $k$, we obtain that 
$(u_k)_k\subset\mM$ is a Palais-Smale sequence for $J$ at level $c$. By Lemma \ref{lem:coercivity}, $(u_k)_k$ is a bounded sequence.
Thus, up to a subsequence, we may assume that $u_k\rightharpoonup u$, weakly in $E$, for some $u\in E$, and the weak sequential
continuity of $J'$ gives $J'(u)=0$. In particular, if $u\neq 0$ then $u\in\mM$ holds, and since $J$ is lower semicontinuous on $\mM$
we obtain
\[ c\leq J(u)\leq \liminf_{k\to\infty} J(u_k) =c.\]
On the other hand, if $u=0$ holds, then we can find $(y_k)_k\subset\R^N$ and $\delta>0$ such that
\begin{equation}\label{eqn:liminf_delta_f}
  \liminf_{k\to\infty}\int_{B_1(0)} (y_k\ast u_k)^2\, dx \geq \delta>0.
\end{equation}
Indeed, if this were false, the concentration-compactness Lemma \cite[Lemma I.1]{LIONS84_2}
would imply $\|u_k\|_{L^p}\to 0$ as $k\to\infty$ and, since 
\[c=\llim_{k\to\infty} J(u_k)=\llim_{k\to\infty}\int_{\R^N}\frac12 f(x,u_k)u_k-F(x,u_k)\, dx 
  \leq \eps \supl_{k\in\N}\|u_k\|_{L^2}^2 + C_\eps \llim_{k\to\infty}\|u_k\|_{L^p}^p,\]
holds for all $\eps>0$, this would contradict the fact that $c>0$.

Now, we also remark that $(y_k)_k$ must be unbounded, since we are assuming $u_k\rightharpoonup 0$.
Hence, passing to a subsequence if necessary, we may suppose $|y_k|\to\infty$ and $y_k\ast u_k\rightharpoonup w$ as $k\to\infty$. 
The compact embedding $H^1(B_1(0))\hookrightarrow L^2(B_1(0))$ then implies $w\neq 0$.
For every $t_k>0$, we have by Lemma \ref{lem:Nehari_f} (ii)
\begin{align*}
  J(u_k)\geq J(t_ku_k) &= J_\infty(t_k(y_k\ast u_k)) +\frac{t_k^2}{2}\int_{\R^N}(a(x+y_k)-a_\infty)(y_k\ast u_k)^2\, dx \\
  &\quad +\int_{\R^N}F(x+y_k,t_k(y_k\ast u_k))-F_\infty(t_k(y_k\ast u_k))\, dx.
\end{align*}
Choosing $t_k>0$ such that $t_k(y_k\ast u_k)\in\mM_\infty$ holds, it follows that $(t_k)_k$ is a bounded sequence, since 
$y_k\ast u_k\rightharpoonup w\neq 0$ (compare with \cite[Proposition 2.7]{SW2009}.)
Since $a(x+y_k)\to a_\infty$ and $y_k\ast u_k\rightharpoonup w$ as $k\to\infty$,
the dominated convergence theorem gives $\int_{\R^N}(a(x+y_k)-a_\infty)(y_k\ast u_k)^2\, dx\to 0$, as $k\to\infty$.
Moreover since $t_ku_k\rightharpoonup 0$ as $k\to\infty$, (A1) and (F2) imply $\int_{\R^N}F(x,t_k u_k)-F_\infty(t_k u_k)\, dx\to 0$, 
as $k\to\infty$.
Hence, we conclude that
\[c=\lim_{k\to\infty} J(u_k)\geq \limsup_{k\to\infty} J_\infty(t_k (y_k\ast u_k))\geq c_\infty\]
holds, which contradicts our assumption $c<c_\infty$ and gives the desired conclusion.
\end{proof}

\section{The existence of a nontrivial solution}
We now consider the case where $0\notin\sigma(-\Delta+a)$ holds, and wish to prove the existence
of a solution to \eqref{eqn:nls_f} under the conditions of Theorem \ref{thm:exist}. We shall assume
throughout this section that $N\geq 2$, (A1), (F1), (F3)--(F5) and (F2$'$) are satisfied for some $\nu>0$. In addition,
the ground-state solution of \eqref{eqn:nls_f_inf} will be required to be unique (up to translations).

The proof of Theorem~\ref{thm:exist} will rely on a topological degree
argument applied to a barycenter type map. In order to set up a
corresponding minimax principle which avoids noncompactness of
Palais-Smale sequences, we first need some asymptotic
estimates.

\subsection{Asymptotic estimates}
The properties (A1), (F1)--(F4) and the oddness of $f_\infty$ ensure the existence of a positive ground-state solution 
to the limit problem \eqref{eqn:nls_f_inf}. More precisely, according to \cite[Th\'eor\` eme 1]{BeGaKa83},  
\cite[Theorem 1]{Berest_Lions83} and \cite[Theorem 2]{GNN}, there exists a ground-state solution $u_\infty\in\mC^2(\R^N)$ of 
\eqref{eqn:nls_f_inf}, positive, radially symmetric, radially
decreasing and satisfying the following exponential decay property:
\begin{equation}\label{eqn:decay_u_infty}
  \llim_{|x|\to\infty}u_\infty(x)|x|^{\frac{N-1}{2}} e^{\sqrt{a_\infty}|x|} \text{ exists and is positive}
\end{equation}
We recall a result of \cite{BaLi90} which we shall use repeatedly in the sequel.
\begin{prop}(\cite[Proposition 1.2]{BaLi90})\label{prop:1.2_BL}
  Let $\varphi\in\mC(\R^N)\cap L^\infty(\R^N)$, $\psi\in\mC(\R^N)$ radially symmetric, satisfy for some
  $\sigma \geq 0, \beta\geq 0$, $\gamma\in\R$
  \begin{align*}
   &\varphi(x)|x|^\beta e^{\sigma|x|} \lra\gamma \quad\text{ as }\quad|x|\to\infty \\
   \text{and } & \int_{\R^N}|\psi(x)|(1+|x|^\beta)e^{\sigma|x|}\, dx <\infty.
  \end{align*}
  Then, as $|y|\to\infty$, there holds
  \[\left(\int_{\R^N}(y\ast\varphi)\psi\, dx\right) |y|^\beta e^{\sigma|y|}  \lra \gamma\int_{\R^N}\psi(x)\exp(-\sigma x_1)\, dx.\]
\end{prop}

An immediate consequence of this proposition and the
estimate \eqref{eqn:eigen_neg} on the eigenfunctions $e_i$ is the existence
of some constant $\kappa_1>0$ such that
\begin{equation}\label{eqn:asympt1}
 \int_{\R^N}(y\ast u_\infty) |e_i|\, dx\leq \kappa_1 |y|^{-\frac{N-1}{2}} e^{-\sqrt{a_\infty}|y|} 
 \qquad \text{for all $i=1,\ldots, n$ and $|y|\geq 1$.}
\end{equation}

In the next result, we consider a convex combination of two translates of the ground-state solution $u_\infty$ and its projection on
$\mM$. We derive estimates concerning its behavior as these translates are moved far apart from each other and far away from the origin. 
The outcome of this study will be used to show that the energy of such a convex combination can be made smaller than $2c_\infty$ under
suitable conditions (see Lemma \ref{lem:energy}.) We introduce the following notation which will be used in the next two lemmata.
For $y,z\in\R^N$, we let $\di{y,z}:=\min\{|y|,|z|,|y-z|\}$, where $|\cdot|$ denotes the Euclidean norm on $\R^N$.
\begin{lem}\label{lem:S2}
  \begin{itemize}
    \item[(i)] There exists $S_1\geq 1$ such that $(1-s)(y\ast u_\infty)+s(z\ast u_\infty)\notin E^-$, for all $y,z\in\R^N $ with 
     $\di{y,z}\geq S_1$ and all $s\in[0,1]$.
    \item[(ii)] For $y, z\in\R^N$ with $\di{y,z}\geq S_1$, and $s\in[0,1]$, let $t_\infty=t_\infty(s,y,z)>0$ and 
     $h_\infty=h_\infty(s,y,z)\in E^-$ be chosen such that 
     \[\hat{\m}((1-s)(y\ast u_\infty)+s(z\ast u_\infty))=t_\infty[(1-s)(y\ast u_\infty) +s(z\ast u_\infty)+ h_\infty]\] holds.
     Then
     \begin{equation}\label{eqn:t_bounded_limit}
      0<\inf_{\substack{\di{y,z}\geq S_1\\s\in[0,1]}}t_\infty(s,y,z)\leq \sup_{\substack{\di{y,z}\geq S_1 \\s\in[0,1]}}t_\infty(s,y,z) <+\infty
     \end{equation}
     and there exists $\kappa_2>0$ such that 
     \begin{equation}\label{eqn:h_bounded_limit}
       \sup_{s\in[0,1]}\|h_\infty(s,y,z)\| \leq \kappa_2 \,\max\{|y|^{-\frac{N-1}{2}},|z|^{-\frac{N-1}{2}}\}\, e^{-\sqrt{a_\infty}\min\{|y|,|z|\}}
     \end{equation}
     for $y,z \in \R^N$ with $\di{y,z} \ge S_1$.\\ 
     Furthermore, if $((s_k,y_k,z_k))_k\subset[0,1]\times\R^N\times\R^N$ satisfies
     $\llim_{k\to\infty}s_k=s\in[0,1]$ and $\llim_{k\to\infty}\di{y_k,z_k}=\infty$, there exists some 
     $T=\llim_{k\to\infty}t_\infty(s_k,y_k,z_k)>0$ such that
     \begin{equation}\label{eqn:J_infty} 
       J(\hat{\m}((1-s_k)(y_k\ast u_\infty) +s_k(z_k\ast u_\infty)))\to J_\infty((1-s)Tu_\infty)+ J_\infty(sTu_\infty),
     \end{equation}
     as $k\to\infty$.
     Moreover, $T=T(s)$ is uniquely determined by the relation
     \begin{equation}\label{eqn:t_tilde}
       \int_{\R^N} [((1-s)^2+s^2)Tf_\infty(u_\infty)-(1-s)f_\infty((1-s)Tu_\infty)-sf_\infty(sTu_\infty)]u_\infty\, dx=0.
     \end{equation}
  \end{itemize}
\end{lem}

\begin{proof}
\begin{itemize}
  \item[(i)] Since $a(x)\to a_\infty$ as $|x|\to\infty$ and $((z-y)\ast u_\infty)(x) \to 0$ as $|z-y|\to\infty$ for all $x\in\R^N$, 
  the dominated convergence theorem implies that
  \begin{align}
    \int_{\R^N} |\nabla[(1-s)(y\ast u_\infty) &+ s(z\ast u_\infty)]|^2 + a(x) [(1-s)(y\ast u_\infty)+s(z\ast u_\infty)]^2\, dx\nonumber\\
    &\lra ((1-s)^2+s^2) \int_{\R^N}|\nabla u_\infty|^2+a_\infty u_\infty^2\, dx \label{eqn:A_Teil_konv}
  \end{align}
  as $\di{y,z}\to \infty$.
  Since $\int_{\R^N}|\nabla h|^2+a(x) h^2\, dx =-\|h\|^2\leq 0$ for all $h\in E^-$, the conclusion follows from
  \eqref{eqn:A_Teil_konv} and the fact that $(1-s)^2+s^2\geq \frac12$ for all $s\in[0,1]$.
  
  \item[(ii)] We set $w_\infty:=(1-s)(y\ast u_\infty) + s(z\ast u_\infty)$. Since
  $J'(t_\infty [w_\infty+ h_\infty])(t_\infty [w_\infty+ h_\infty])=0$, we find
  \begin{align*}
    0< 2c &\leq 2J(t_\infty [w_\infty+h_\infty])+2\int_{\R^N}F(x,t_\infty [w_\infty+h_\infty])\, dx\\
    &= t_\infty^2 \int_{\R^N} |\nabla[w_\infty+h_\infty]|^2 + a(x) [w_\infty+h_\infty]^2\, dx\\
    &\leq t_\infty^2 \left\{ \int_{\R^N}|\nabla w_\infty|^2+a(x) w_\infty^2\, dx
    + 4C\|u_\infty\|\, \|h_\infty\| -\|h_\infty\|^2\right\}.
  \end{align*}
  Using \eqref{eqn:A_Teil_konv}, we deduce that
  $\supl_{\substack{\di{y,z}\geq S_1\\ s\in[0,1]}} \|h_\infty(s,y,z)\|<+\infty$ and \\
  $\infl_{\substack{\di{y,z}\geq S_1\\ s\in[0,1]}}t_\infty(s,y,z)>0.$
  Now suppose by contradiction, that $t_\infty$ is not bounded above, and let
  $((s_k,y_k,z_k))_k$ $\subset[0,1]\times\R^N\times\R^N$ satisfy $\di{y,z}\geq S_1$ for all $k$
  as well as $\llim_{k\to\infty}t_\infty(s_k,y_k,z_k)=\infty$.
  Up to a subsequence, we may assume $s_k\to s\in[0,1]$, and the continuity of $\hat{\m}$ then
  implies $\llim_{k\to\infty}|y_k|=\infty$ and $s\neq 1$, or $\llim_{k\to\infty}|z_k|=\infty$ and $s\neq 0$, up
  to a subsequence.

  We consider the case $s\neq 1$ and $|y_k|\to\infty$ as $k\to\infty$; the other case follows similarly.
  Writing $h_\infty=\suml_{i=1}^n A_i^\infty e_i$ with $A_1^\infty, \ldots, A_n^\infty\in\R$, the upper bound on $\|h_\infty\|$
  and the decay property \eqref{eqn:eigen_neg} of the eigenfunctions $e_1,\ldots, e_n$ imply $h_\infty(x+y_k)\to 0$ as 
  $k\to\infty$ for all $x\in\R^N$. Since $u_\infty>0$ holds on $\R^N$, we have for all $x\in\R^N$
  \[t_\infty[w_\infty(x+y_k)+h_\infty(x+y_k)] \geq t_\infty[(1-s_k)u_\infty(x)+h_\infty(x+y_k)]\to+\infty,\]
  as $k\to\infty.$
  The assumption (F3) and Fatou's Lemma then give
  \begin{align*}
    &\int_{\R^N}|\nabla[w_\infty(x)+h_\infty(x)]|^2 +a(x)[w_\infty(x)+h_\infty(x)]^2\, dx\\
    &=\frac{1}{t_\infty}\int_{\R^N}f(x,t_\infty[w_\infty(x) +h_\infty(x)])(w_\infty(x)+h_\infty(x))\, dx\\
    &> 2\int_{\R^N}\frac{F(x+y_k,t_\infty[w_\infty(x+y_k)+h_\infty(x+y_k)])}{(t_\infty[w_\infty(x+y_k)+h_\infty(x+y_k)])^2}
    (w_\infty(x+y_k)+h_\infty(x+y_k))^2\, dx\\
    &\to +\infty, \quad \text{ as }k\to\infty,
  \end{align*}
  which contradicts the boundedness of $w_\infty+h_\infty$ and concludes the proof of \eqref{eqn:t_bounded_limit}.

  Now the property $J'(t_\infty[w_\infty+ h_\infty])h_\infty=0$ and the inequality 
  $f(x,u+v)v\geq f(x,u)v$ for all $x\in\R^N$, $u,v\in\R$, which follows from (F4), together give
  \begin{equation}\label{eqn:h_tilde_infty_f}
    \|h_\infty\|^2 \leq \int_{\R^N}\nabla w_\infty\cdot\nabla h_\infty +a(x)w_\infty h_\infty\, dx 
    + t_\infty^{-1}\int_{\R^N}|f(x,t_\infty w_\infty) h_\infty|\, dx.
  \end{equation}
  We write again $h_\infty=\suml_{i=1}^n A_i^\infty e_i$, so that
  \begin{equation}\label{eqn:equiv_norm_h}
     |\lambda_n| \sum_{i=1}^n (A_i^\infty)^2 \leq \|h_\infty\|^2=
     -\sum_{i=1}^n \lambda_i (A_i^\infty)^2 \leq |\lambda_1|\sum_{i=1}^n (A_i^\infty)^2.
  \end{equation}
  Then, using \eqref{eqn:estim_f_eps} and the facts that 
  $u_\infty\in L^\infty(\R^N)$ and $w_\infty$ is positive, we infer from
  \eqref{eqn:h_tilde_infty_f} and \eqref{eqn:equiv_norm_h} that
  \begin{align*}
    \|h_\infty\|^2 &\leq \int_{\R^N} w_\infty \,|\sum_{i=1}^n \lambda_i A_i^\infty e_i|\, dx
    + C'  \int_{\R^N} w_\infty |\sum_{i=1}^n A_i^\infty e_i|\,dx\\ 
    &\leq C'' \|h_\infty\| \max_{1\leq i\leq n}\int_{\R^N}w_\infty | e_i|\,dx,
  \end{align*}
  with constants $C',C''>0$, and hence
  \begin{align*}
    \|h_\infty\| \leq C''  \max_{1\leq i\leq n}\int_{\R^N} w_\infty |e_i|\,dx \le
    C''  \max_{1\leq i\leq n}\int_{\R^N} (y* u_\infty+z*u_\infty)|e_i|\,dx\\
    \le 2C'' \kappa_1 \max\{|y|^{-\frac{N-1}{2}},|z|^{-\frac{N-1}{2}}\}\, e^{-\sqrt{a_\infty}\min\{|y|,|z|\}}
  \end{align*}
  by \eqref{eqn:asympt1}. This proves \eqref{eqn:h_bounded_limit} with
  $\kappa_2= 2C'' \kappa_1$.

  Let now $((s_k,y_k,z_k))_k\subset[0,1]\times\R^N\times\R^N$ and $s\in[0,1]$ be such that $\llim_{k\to\infty}s_k=s$, and 
  $\llim_{k\to\infty}\di{y_k,z_k}=\infty$ hold.
  By \eqref{eqn:t_bounded_limit}, we can assume that, up to a subsequence,
  $\llim_{k\to\infty}t_\infty(s_k,y_k,z_k)=T>0$ holds. Consequently, we find
  \begin{align*}
    0 &=\lim_{k\to\infty}J'(t_\infty[w_\infty+h_\infty])t_\infty[w_\infty+h_\infty]\\
    &=((1-s)^2+s^2)T^2\int_{\R^N}|\nabla u_\infty|^2 +a_\infty u_\infty^2\, dx\\
    &- \int_{\R^N}f_\infty((1-s)Tu_\infty)(1-s)Tu_\infty\, dx 
    - \int_{\R^N}f_\infty(sTu_\infty)sTu_\infty\, dx.
  \end{align*}
  Since $J_\infty'(u_\infty)u_\infty=0$, we conclude that $T$ satisfies \eqref{eqn:t_tilde}. 
  The strict monotonicity in (F4) and the fact that $u_\infty>0$ on $\R^N$ give the uniqueness of $T$ 
  (recall that $s\in[0,1]$ is fixed). Hence the whole sequence $(t_\infty(s_k,y_k,z_k))_k$ converges towards $T$ 
  and the latter is uniquely determined by $s=\llim_{k\to\infty}s_k$.

  To prove \eqref{eqn:J_infty}, remark that
  \begin{equation*}
    \lim_{k\to\infty}\int_{\R^N}F(x,t_\infty[w_\infty+h_\infty])\, dx
    =\int_{\R^N}F_\infty((1-s)Tu_\infty)\, dx +\int_{\R^N}F_\infty(sTu_\infty)\,dx
  \end{equation*}
  holds. Hence, similar arguments as above imply
  $\llim_{k\to\infty}J(t_\infty[w_\infty+h_\infty])=J_\infty((1-s)Tu_\infty) + J_\infty(sTu_\infty)$
  which concludes the proof.
\end{itemize}
\end{proof}

\begin{remq}\label{remq:A_infty}
Taking $s=0$ in the above lemma, we find $\llim_{|y|\to\infty}J(\hat{\m}(y\ast u_\infty))=c_\infty$. In particular,
$c\leq c_\infty$ holds.
\end{remq}

The following result is crucial for the construction of the min-max value below. 
We now work under the additional assumption
\eqref{eqn:thm1.1:a_f} of Theorem \ref{thm:exist}. Furthermore,
we may assume that 
\begin{equation}\label{eq:12}
 2<\alpha<p   
\end{equation}
holds in \eqref{eqn:thm1.1:a_f}.

\begin{lem}[Energy estimate]\label{lem:energy}
  There exists $S_2\geq \frac32 S_1$ such that
  \[ J(\hat{\m}((1-s)(y\ast u_\infty) + s(z\ast u_\infty)))<2c_\infty\]
  holds for all $s\in[0,1]$, $R\geq S_2$ and $y,z\in\R^N$ 
  with $|y| \ge R$, $|z|\geq R$ and $\frac23 R\leq |y-z|\leq 2R$.
\end{lem}
\begin{proof}
Throughout the proof, we consider 
\begin{equation}\label{eq:7}
 R \ge \frac{3}{2} S_1\qquad \text{and} \qquad y,z \in \R^N \quad \text{with $|y| \ge R$, $|z|\geq R$ and $\frac23 R\leq |y-z|\leq 2R$.}  
\end{equation}
For such $y,z$ and $s\in[0,1]$, we set $w_\infty=(1-s)(y\ast u_\infty)+s(z\ast u_\infty)$ and choose $t_\infty, h_\infty$ as in 
Lemma \ref{lem:S2} (ii). We emphasize that $w_\infty$, $t_\infty$ and $h_\infty$ depend
in a crucial way on $y,z$, but we suppress this dependence in our
notation. All constants in the following will neither depend on $R$
nor on $s,y,z$. In the sequel, we will bound terms relative to the
{\em asymptotic exchange energy}  
\[d_{y,z}:= \int_{\R^N} f_\infty(y*u_\infty) (z * u_\infty)\,dx=\int_{\R^N} f_\infty(u_\infty) (z-y)* u_\infty\,dx.\]
We first collect a few easy consequences of
Proposition~\ref{prop:1.2_BL}. First, since by (F2$'$) and \eqref{eqn:decay_u_infty} we have 
\[\int_{\R^N} f_\infty(u_\infty(x))e^{\sqrt{a_\infty} |x|}(1+|x|^{\frac{N-1}{2}})\,dx < \infty,\]
Proposition~\ref{prop:1.2_BL} implies that there is $\kappa_3>0$ such that 
\begin{equation}\label{eq:8}
  \frac{1}{\kappa_3}|y-z|^{-\frac{N-1}{2}}e^{-\sqrt{a_\infty}|y-z|}
  \le d_{y,z} \le \kappa_3 |y-z|^{-\frac{N-1}{2}}e^{-\sqrt{a_\infty}|y-z|}  
\end{equation}
for $R,y,z$ satisfying \eqref{eq:7}.
Moreover, by making $\kappa_3$ larger if necessary, we may also assume
that 
\begin{equation}\label{eq:9}
  \max \{|y|^{1-N},|z|^{1-N}\}e^{-2 \sqrt{a_\infty}\min\{|y|,|z|\}} \le \kappa_3 R^{-\frac{N-1}{2}} d_{y,z}
\end{equation}
for $R,y,z$ satisfying \eqref{eq:7}. Now let $\alpha>2$ be as in
assumption \eqref{eqn:thm1.1:a_f}. Applying
Proposition~\ref{prop:1.2_BL} to $\varphi= u_\infty^2$, $\psi(x)=e^{-\alpha\sqrt{a_\infty}|x|}$, $\sigma=2 \sqrt{a_\infty}$
and $\beta=N-1$, we obtain 
\[\int_{\R^N} e^{-\alpha\sqrt{a_\infty}|x|}\{(y\ast u_\infty)^2 + (z\ast u_\infty)^2\}\, dx 
  \le C \max \{|y|^{1-N},|z|^{1-N}\}\,e^{-2 \sqrt{a_\infty}\min\{|y|,|z|\}}\]
for all $y,z \in \R^N$ with some constant $C>0$. Therefore by
\eqref{eq:9} we have,
making $\kappa_3$ larger if necessary, 
\begin{equation}\label{eq:10}
  \int_{\R^N} e^{-\alpha\sqrt{a_\infty}|x|}\{(y\ast u_\infty)^2 + (z\ast u_\infty)^2\}\, dx \le \kappa_3 R^{-\frac{N-1}{2}} d_{y,z}  
\end{equation}
for all $R,y,z$ satisfying \eqref{eq:7}. Moreover, taking \eqref{eq:12} into account, and applying Proposition~\ref{prop:1.2_BL} to
$\varphi(x)=e^{-\alpha\sqrt{a_\infty}|x|}$, $\psi=u_\infty^p$, $\sigma=\alpha \sqrt{a_\infty}$ and
$\beta=\alpha\frac{N-1}{2}$, we find that 
\begin{align*}
  &\int_{\R^N} e^{-\alpha\sqrt{a_\infty}|x|}\{(y\ast u_\infty)^p + (z\ast u_\infty)^p\}\, dx = \int_{\R^N} [(-y)*\varphi+(-z)* \varphi]\psi \,dx\\ 
  &\le C' \max \{|y|^{-\alpha \frac{N-1}{2}},|z|^{-\alpha \frac{N-1}{2}}\}\,e^{-\alpha \sqrt{a_\infty}\min\{|y|,|z|\}}\\
  &\le C' \max \{|y|^{1-N},|z|^{1-N}\}\,e^{-2 \sqrt{a_\infty}\min\{|y|,|z|\}}
\end{align*}
for all $y,z \in \R^N$, $|y|, |z|\geq 1$, with some constant $C'>0$. Therefore by 
\eqref{eq:9} we have, 
making $\kappa_3$ again larger if necessary, 
\begin{equation} \label{eq:13}
  \int_{\R^N} e^{-\alpha\sqrt{a_\infty}|x|}\{(y\ast u_\infty)^p + (z\ast u_\infty)^p\}\, dx \le \kappa_3 R^{-\frac{N-1}{2}} d_{y,z}  
\end{equation}
for all $R,y,z$ satisfying \eqref{eq:7}. Finally, let $\nu>0$ be as in
assumption (F2$'$). Then applying Proposition~\ref{prop:1.2_BL} to
$\varphi=\psi= u_\infty^{1+\frac{\nu}{2}}$, $\alpha= \sqrt{a_\infty}$ and $\beta= N-1$ yields 
\[\int_{\R^N} (y\ast u_\infty)^{1+\frac{\nu}{2}}(z\ast u_\infty)^{1+\frac{\nu}{2}}\, dx=
  \int_{\R^N} ((y-z) \ast u_\infty)^{1+\frac{\nu}{2}} u_\infty^{1+\frac{\nu}{2}}\, dx \le C'' |y-z|^{1-N}e^{-\sqrt{a_\infty}|y-z|} \]
for all $y,z \in \R^N$ with some constant $C''>0$. Therefore, making $\kappa_3$ again larger if necessary, 
\begin{equation} \label{eq:14}
  \int_{\R^N} (y\ast u_\infty)^{1+\frac{\nu}{2}}(z\ast u_\infty)^{1+\frac{\nu}{2}}\, dx \le \kappa_3 R^{-\frac{N-1}{2}} d_{y,z}    
\end{equation}
holds for all $R,y,z$ satisfying \eqref{eq:7}. We now have all the tools to estimate
\[ J(\hat\m(w_\infty)) = \frac{t_\infty^2}{2}\int_{\R^N}|\nabla(w_\infty+h_\infty)|^2+a(x)[w_\infty+h_\infty]^2\, dx
  -\int_{\R^N}F(x,t_\infty[w_\infty+h_\infty]), dx.\]
We start by estimating the first integral on the right-hand side which we split in the following way.
\begin{align*}
  &\int_{\R^N}|\nabla(w_\infty+h_\infty)|^2+a(x)[w_\infty+h_\infty]^2\, dx\\
  &\quad =((1-s)^2+s^2)\int_{\R^N}|\nabla u_\infty|^2+a_\infty u_\infty^2\, dx\\
  &\quad+ 2(1-s)s\int_{\R^N}\nabla(y\ast u_\infty)\cdot\nabla(z\ast u_\infty)+a_\infty (y\ast u_\infty)(z\ast u_\infty)\, dx\\
  &\quad+\int_{\R^N}(a(x)-a_\infty)w_\infty^2\, dx + 2\int_{\R^N}\nabla w_\infty\cdot\nabla h_\infty+a(x) w_\infty h_\infty\, dx -\|h_\infty\|^2.
\end{align*}
The property $J_\infty'(u_\infty)=0$ implies that
\begin{equation}\label{eqn:Jprime_infty}
  \int_{\R^N}\!\!\!\nabla(y\ast u_\infty)\cdot\nabla(z\ast u_\infty)+a_\infty (y\ast u_\infty)(z\ast u_\infty)\, dx  
  =\int_{\R^n}\!\!f_\infty(y\ast u_\infty)(z\ast u_\infty)= d_{y,z}.
\end{equation}
We deduce from \eqref{eq:10} and condition \eqref{eqn:thm1.1:a_f} that 
\begin{align*}
  &\int_{\R^N}(a(x)-a_\infty) w_\infty^2\, dx 
   \leq 2C_1 \int_{\R^N}[(y\ast u_\infty)^2 + (z\ast u_\infty)^2]e^{-\alpha\sqrt{a_\infty}|x|}\, dx\\
  &\leq 2 C_1 \kappa_3 R^{-\frac{N-1}{2}} d_{y,z}\qquad \text{for all $s\in[0,1]$ and $R,y,z$ satisfying \eqref{eq:7}.}
\end{align*}
As in the proof of Lemma \ref{lem:S2}, we write $h_\infty=\suml_{i=1}^n A_i^\infty e_i$. 
Using \eqref{eqn:asympt1}, \eqref{eqn:h_bounded_limit},
\eqref{eqn:equiv_norm_h} and \eqref{eq:9}, we obtain
\begin{align*}
  &\int_{\R^N}\nabla w_\infty\cdot\nabla h_\infty+a(x) w_\infty h_\infty\, dx
  \leq \left(\sum_{i=1}^n(A_i^\infty)^2\right)^{\frac12}\left(\sum_{i=1}^n \lambda_i^2
  \left(\int_{\R^N}w_\infty |e_i|\, dx\right)^2\right)^{\frac12}\\
  &\leq \kappa_4 \max\{|y|^{1-N},|z|^{1-N}\}e^{-2\sqrt{a_\infty}\min\{|y|,|z|\}}
  \leq \kappa_4 \kappa_3 R^{-\frac{N-1}{2}}d_{y,z}
\end{align*}
for $s\in[0,1]$ and $R,y,z$ satisfying \eqref{eq:7}. Here $\kappa_4:= \frac{2\sqrt{n}|\lambda_1|}{\sqrt{|\lambda_n|}}\kappa_1 \kappa_2$. 
Turning to the second integral, we write
\begin{align*}
  &\int_{\R^N}F(x,t_\infty[w_\infty+h_\infty])\, dx = \int_{\R^N} F_\infty(st_\infty  u_\infty)+F_\infty((1-s)t_\infty u_\infty)\, dx\\
  &+\int_{\R^N}F_\infty(t_\infty w_\infty) -[ F_\infty(st_\infty(y\ast u_\infty))+F_\infty((1-s)t_\infty (z\ast u_\infty)) ]\, dx \\
  &+\int_{\R^N}F(x,t_\infty w_\infty)-F_\infty(t_\infty w_\infty)\, dx+\int_{\R^N}F(x, t_\infty[w_\infty+h_\infty])-F(x,t_\infty w_\infty)\, dx.
\end{align*}
From \eqref{eqn:ineq_Finf_finf} and \eqref{eq:14}, it follows that
\newpage
\begin{align*}
  &\int_{\R^N}F_\infty(t_\infty w_\infty) -[ F_\infty(st_\infty(y\ast u_\infty))+F_\infty((1-s)t_\infty (z\ast u_\infty)) ]\, dx\\
  &\geq (1-s)t_\infty\int_{\R^N} f_\infty (st_\infty(y\ast u_\infty))(z\ast u_\infty)\, dx \\
  &\qquad+ st_\infty \int_{\R^N}f_\infty((1-s)t_\infty (z\ast u_\infty))(y\ast u_\infty)\,dx \\
  &\qquad- C t_\infty^{2+\nu} s^{1+\frac{\nu}{2}}(1-s)^{1+\frac{\nu}{2}}\int_{\R^N} (y\ast u_\infty)^{1+\frac{\nu}{2}}(z\ast u_\infty)^{1+\frac{\nu}{2}}\, dx \\
  &\geq (1-s)t_\infty\int_{\R^N} f_\infty (st_\infty(y\ast u_\infty))(z\ast u_\infty)\, dx \\
  &\qquad+ st_\infty \int_{\R^N}f_\infty((1-s)t_\infty (z\ast u_\infty))(y\ast u_\infty)\,dx  -\kappa_5 R^{-\frac{N-1}{2}}d_{y,z}
\end{align*}
for $s\in[0,1]$ and $R,y,z$ satisfying \eqref{eq:7}. Here $\kappa_5:=
C \kappa_3 \supl_{\substack{\di{y,z}\geq S_1  \\s\in[0,1]}}[t_\infty(s,y,z)]^{2+\nu}$ and $C=C_\rho$ is the constant from \eqref{eqn:ineq_Finf_finf}
corresponding to the value 
\[\rho= \|u_\infty\|_\infty \sup_{\substack{\di{y,z}\geq S_1  \\s\in[0,1]}}t_\infty(s,y,z),\]
which is finite by \eqref{eqn:t_bounded_limit}. Moreover, condition \eqref{eqn:thm1.1:a_f} as well as 
\eqref{eq:10},~\eqref{eq:13} and \eqref{eqn:t_bounded_limit}  imply that 
\begin{align*}
  \int_{\R^N}&F(x,t_\infty w_\infty)-F_\infty(t_\infty w_\infty)\, dx 
  \geq -2C_2 t_\infty^2 \int_{\R^N} e^{-\alpha\sqrt{a_\infty}|x|}\{(y\ast u_\infty)^2 + (z\ast u_\infty)^2\}\, dx\\
  &-2^{p-1} C_2 t_\infty^p \int_{\R^N}  e^{-\alpha\sqrt{a_\infty}|x|}\{(y\ast u_\infty)^p + (z\ast u_\infty)^p\}\, dx
  \geq -\kappa_6 R^{-\frac{N-1}{2}} d_{y,z},
\end{align*} 
for $s\in[0,1]$ and $R,y,z$ satisfying \eqref{eq:7}, where
$\kappa_6>0$ is a constant. From \eqref{eqn:ineq_F_f}, \eqref{eqn:t_bounded_limit},
\eqref{eqn:h_bounded_limit} and \eqref{eqn:equiv_norm_h}, we finally
obtain a constant $\kappa_7>0$ such that 
\begin{align*}
  &\int_{\R^N}F(x, t_\infty(w_\infty+h_\infty))-F(x,t_\infty w_\infty)\, dx
  \geq \int_{\R^N}f(x, t_\infty w_\infty)h_\infty\, dx \\
  &= \sum_{i=1}^n A_i^\infty \int_{\R^N} f(x, t_\infty w_\infty) e_i\,dx
  \geq - \!\!\sup_{\substack{|\tau| \le 2 t_\infty \|u_\infty\|_\infty \\ x \in \R^N}}\!\!\!\!\frac{|f(x,\tau)|}{|\tau|}
  \;t_\infty \sum_{i=1}^n|A_i^\infty| \int_{\R^N} w_\infty |e_i|\, dx\\
  & \geq -\kappa_7 R^{-\frac{N-1}{2}}d_{y,z}
\end{align*}
for $s\in[0,1]$ and $R,y,z$ satisfying \eqref{eq:7}.
Summarizing, we can write
\begin{align}
  J(\hat\m(w_\infty))&\le J_\infty(st_\infty u_\infty) + J_\infty((1-s)t_\infty u_\infty) \label{eq:11}\\
  &+ (s(1-s)t_\infty^2+\kappa_8 R^{-\frac{N-1}{2}})d_{y,z} \nonumber\\
  &-(1-s)t_\infty\int_{\R^N}f_\infty(st_\infty (y\ast u_\infty))(z\ast u_\infty) dx \nonumber\\
  &-st_\infty \int_{\R^N}f_\infty((1-s)t_\infty(z\ast u_\infty))(y\ast u_\infty) dx.\nonumber
\end{align}
with $\kappa_8=\tau^2\kappa_3 (C_1 +\kappa_4) +\sum \limits_{j=5}^7 \kappa_j$ and 
$\tau=\sup\{t_\infty(s,y,z)\,:\, s\in[0,1], \; \di{y,z}\geq S_1\}$. Now, by \eqref{eqn:t_bounded_limit}, we can find some 
$0<\delta_0<1$ such that 
\[J_\infty(st_\infty u_\infty) + J_\infty((1-s)t_\infty u_\infty)+s(1-s)t_\infty^2 d_{y,z}  \leq \frac{3}{2}c_\infty\]
for all $s\in[0,\delta_0)\cup(1-\delta_0,1]$ and $R,y,z$
satisfying \eqref{eq:7}. Hence, by \eqref{eq:11}, there exists 
$R_0 \ge \frac32S_1$ such that 
\begin{equation} \label{eq:17}
  J(\hat{\m}(w_\infty))<2c_\infty \quad \text{for $R\geq R_0$, $y,z$ satisfying \eqref{eq:7}
  and $s\in [0,\delta_0)\cup (1-\delta_0,1]$.}  
\end{equation}
Next, conditions (F4) and (F5) give
\[\int_{\R^N}f_\infty(st_\infty (y\ast u_\infty))(z\ast u_\infty) dx \geq st_\infty\min\{(st_\infty)^\theta, 1\}\,d_{y,z} \]
and 
\[\int_{\R^N}f_\infty((1-s)t_\infty (y\ast u_\infty))(z\ast u_\infty) dx \geq (1-s)t_\infty\min\{((1-s)t_\infty)^\theta, 1\}\,d_{y,z},\]
and therefore \eqref{eq:11} yields 
\begin{align}
  &J(\hat\m(w_\infty))-2c_\infty \le  J(\hat\m(w_\infty))- J_\infty(st_\infty u_\infty) - J_\infty((1-s)t_\infty u_\infty) \label{eq:15}\\
  &\le \Bigl[s(1-s)t_\infty^2\Bigl(1-\min\{(st_\infty)^\theta, 1\}-\min\{((1-s)t_\infty)^\theta, 1\}\Bigr)
  + \kappa_8 R^{-\frac{N-1}{2}}\Bigr] d_{y,z}. \nonumber 
\end{align}
We now claim that there is some $R_1\geq \frac32 S_1$ and some
$\kappa_9>0$ such that
\begin{equation} \label{eq:16}
  s(1-s)t_\infty^2\Bigl(1-\min\{(st_\infty)^\theta, 1\}-\min\{((1-s)t_\infty)^\theta, 1\}\Bigr) <-\kappa_9
\end{equation}
for all $s \in[\delta_0, 1-\delta_0]$, $R\geq R_1$
and $y,z$ satisfying \eqref{eq:7}. For this we consider an arbitrary
sequence $((s_k,y_k,z_k))_k\subset[\delta_0,1-\delta_0]\times\R^N\times\R^N$
such that $\di{y_k,z_k} \to \infty$ and $s_k\to s\in[\delta_0,1-\delta_0]$ as $k\to\infty$. According to Lemma
\ref{lem:S2} we have $\llim_{k\to\infty}t_\infty(s_k,y_k,z_k)=T$ with $T=T(s)$ given by
\eqref{eqn:t_tilde}. Note that $T>0$ by
\eqref{eqn:t_bounded_limit}, and
$\min\{sT,(1-s)T\}\geq \delta_0T>0$. Moreover, $s(1-s)T^2 \ge
(\delta_0T)^2>0$, and \eqref{eqn:t_tilde} implies
$\max\{sT,(1-s)T\}\geq 1$. Consequently
\[s(1-s)T^2\Bigl(1-\min\{(sT)^\theta, 1\}-\min\{((1-s)T)^\theta, 1\}\Bigr) <0,\]
and this shows that \eqref{eq:16} holds for all
$s \in[\delta_0, 1-\delta_0]$, $R\geq R_1$
and $y,z$ satisfying \eqref{eq:7}, where $R_1\geq \frac32 S_1$ and
$\kappa_9>0$ are suitable constants. 
Going back to \eqref{eq:15}, we conclude that  
\[J(\hat\m(w_\infty)) \leq 2c_\infty -  [\kappa_9-\kappa_8 R^{-\frac{N-1}{2}}] d_{y,z},\]
for these values of $R,y,z$ and $s$, and the right hand side of
this inequality is smaller than $2c_\infty$ for $R$ large
enough. Together with \eqref{eq:17} this finishes the proof. 
\end{proof}

We conclude this preparatory section by describing the behavior of the Palais-Smale sequences taken from $\mM$, and show
that the result of Bahri and Lions \cite[Proposition II.1]{BaLions97} (see also \cite[Theorem 8.4]{WILLEM}) holds in our context.
\begin{lem}\label{lem:Benci-Cerami}
  Let $(u_k)_k\subset\mM$ be a sequence for which $(J(u_k))_k$ is bounded and $J'(u_k)\to0$ as $k\to\infty$ holds.
  Then, there exist $\ell\in\N\cup\{0\}$, $(x_k^i)_k\subset\R^N$, $1\leq i\leq \ell$, and $\overline{u}, w_1,\ldots, w_\ell\in E$ satisfying
  (up to a subsequence)
  \begin{itemize}
    \item[(i)] $J'(\overline{u})=0$,
    \item[(ii)] $J_\infty'(w_i)=0$, $i=1, \ldots, \ell$,	
    \item[(ii)] $|x_k^i|\to\infty$ and $|x_k^i-x_k^j|\to\infty$ as $k\to\infty$ for $1\leq i\neq j\leq \ell$,
    \item[(iii)] $\bigl\|u_k-[\overline{u}+\suml_{i=1}^\ell x_k^i\ast w_i]\bigr\|\to 0$ as $k\to\infty$,
    \item[(iv)] $J(u_k) \to J(\overline{u})+\suml_{i=1}^\ell J_\infty(w_i),$ as $k\to\infty$.
  \end{itemize}
\end{lem}
\begin{proof}
From Lemma \ref{lem:coercivity}, $(u_k)_k$ is a bounded sequence in $E$. Up to a subsequence, we may assume 
$u_k\rightharpoonup \overline{u}$ for some $\overline{u}\in E$ and $J(u_k)\to d$ as $k\to\infty$. 
Since $J'$ is weakly sequentially continuous we obtain $J'(\overline{u})=0$.

\textbf{Step 1}: Let $v_k^1:=u_k-\overline{u}$ for all $k\in\N$. Since $v_k^1\rightharpoonup 0$ in $E$ and  
  $a(x)\to a_\infty$ for $|x|\to\infty$, the compactness 
  of the embedding $H^1(B_R(0))\hookrightarrow L^2(B_R(0))$ for all $R>0$ implies
  \begin{equation}\label{eq:int:ainf} 
    \int_{\R^N}(a(x)-a_\infty) |v_k^1|^2\, dx \lra 0, \quad\text{ as }k\to\infty.
  \end{equation}
  Moreover, from (A1) and (F2), it follows that
  \begin{equation}\label{eq:int:qinf} 
    \int_{\R^N}F(x,v_k^1)-F_\infty(v_k^1)\, dx \lra 0, \quad\text{ as }k\to\infty.
  \end{equation}
  Consequently, as $k\to\infty$, there holds
  $J_\infty(v^1_k)=J(v^1_k) + o(1)$
  \[= J(u_k) - J(\overline{u}) +\int_{\R^N}[F(x,u_k)-F(x,\overline{u})-F(x,v_k^1)]\, dx + o(1)
    = J(u_k)-J(\overline{u})+o(1),\]
  where the last step follows from Proposition \ref{prop:decomposition}. (Remark that $|\overline{u}(x)|\to 0$ as $|x|\to\infty$, since
  $J'(\overline{u})=0$. See \cite[Lemma 1]{PANKOV08}.)
  For every $\varphi\in E$, we have furthermore
  \begin{align*}
    J_\infty'(v^1_k)\varphi &= \int_{\R^N}\nabla v_k^1\cdot\nabla \varphi + a_\infty v^1_k \varphi\, dx - \int_{\R^N} f_\infty(v^1_k)\varphi\, dx\\
    &= J'(u_k)\varphi - J'(\overline{u})\varphi +\int_{\R^N}(a_\infty-a(x))v_k^1 \varphi\, dx + \int_{\R^N} [f(x,v_k^1)-f_\infty(v_k^1)]\varphi\, dx\\
    &+\int_{\R^N} [ f(x,u_k)-f(x,\overline{u})-f(x,v_k^1)]\varphi\, dx.
  \end{align*}
  Since $\llim_{k\to\infty}J'(u_k)= 0$ in $H^{-1}$ and $J'(\overline{u})=0$, similar arguments as above (using again Proposition \ref{prop:decomposition}) imply
  \[J'_\infty(v^1_k)\to 0 \;\text{ in }H^{-1},\quad \text{ as }k\to\infty.\]

\textbf{Step 2}: Let \[\zeta:=\limsup_{k\to\infty}\left(\sup_{y\in\R^N}\int_{B_1(y)}|v_k^1|^2\, dx\right).\]
  If $\zeta=0$, then \cite[Lemma II.1]{LIONS84_2} gives $\|v_k^1\|_{L^p}\to 0$ as $k\to\infty$, and from
  \begin{align*}
    \|v_k^1\|^2 &\leq C \int_{\R^N} |\nabla v_k^1|^2 + a_\infty (v_k^1)^2\, dx = C \{ J_\infty'(v_k^1)v_k^1 + \int_{\R^N} f_\infty(v_k^1)v_k^1\, dx\},\\
    &\leq C\{ J_\infty'(v_k^1)v_k^1 + \eps\|v_k^1\|_{L^2}^2+C_\eps\|v_k^1\|_{L^p}^p\}\quad \text{ for all }\eps>0,
  \end{align*}
  we obtain $\|v_k^1\|\to 0$, and hence $u_k\to \bar{u}$ in $E$, as $k\to\infty$, and the proof is complete.
  
  On the other hand, if $\zeta>0$, then, passing to a subsequence, we can find a sequence $(x_k^1)_k\subset\R^N$ satisfying
  $|x_k^1|\to\infty$ as $k\to\infty$ and
  \[\int_{B_1(0)} [(-x_k^1)\ast v_k^1]^2\, dx=\int_{B_1(x_k^1)}(v_k^1)^2\, dx > \frac{\zeta}{2}\]
  for all $k$. Since $((-x_k^1)\ast v_k^1)_k$ is bounded in $E$, passing to a further subsequence, we find 
  $(-x_k^1)\ast v_k^1\rightharpoonup w_1\neq 0$, using the compactness of the embedding $H^1(B_1(0))\hookrightarrow L^2(B_1(0))$.
  Since $J_\infty'$ is weakly sequentially continuous and invariant under translation, there holds $J_\infty'(w_1)
  =\llim_{k\to\infty}J_\infty'((-x_k^1)\ast v_k^1)=\llim_{k\to\infty}J_\infty'(v_k^1)=0\in H^{-1}$.

  Setting $v_k^2:= v_k^1-(x_k^1\ast w_1)$, we obtain $v_k^2\rightharpoonup 0$, and the same arguments as above, applied to $J_\infty$, give
  \[J_\infty(v_k^2)= J_\infty(v_k^1) -J_\infty(w_1)+o(1)=J(u_k)-J(\bar{u})-J_\infty(w_1)+o(1),\]
  $J_\infty'(v_k^2)\to 0$ in $H^{-1}$ and $\|v_k^2\|^2 = \|u_k\|^2 -\|\bar{u}\|^2-\|w_1\|^2 + o(1)$ as $k\to\infty$.

Iterating this procedure, we construct sequences $(x_k^i)_k\subset\R^N$ such that $|x_k^i|\to\infty$ and $|x_k^i-x_k^j|\to\infty$ for all $i\neq j$, as $k\to\infty$,
and critical points $w_i$ of $J_\infty$ such that 
$J_\infty(u_k-\overline{u}-\suml_{j=1}^ix_k^j\ast w_j)= J(u_k)-J(\bar{u})-\suml_{j=1}^i J_\infty(w_j)+o(1)$, and $J_\infty'(u_k-\overline{u}-\suml_{j=1}^ix_k^j\ast w_j)\to 0$ in $H^{-1}$ 
as $k\to\infty$.
Since $J_\infty(w)\geq c_\infty>0$ holds for every critical point $w$ of $J_\infty$,
and since $(J(u_k))_k$ is bounded, the procedure has to stop after a finite number of steps.
\end{proof}

\subsection{The proof of Theorem \ref{thm:exist}}
Suppose all the assumptions of Theorem \ref{thm:exist} are satisfied. We shall prove the existence
of a nontrivial solution to \eqref{eqn:nls_f} in three steps. First note that $c\leq c_\infty$ holds, by Remark \ref{remq:A_infty}.
If $c<c_\infty$ then Proposition \ref{prop:c<cinf_f} gives the desired conclusion. Hence, we can assume that $c=c_\infty$ holds.

Now, consider the barycenter function
$\beta$: $E\backslash\{0\}$ $\to$ $\R^N$ given by
\[\beta(u) = \frac{1}{\|u\|_{L^p}^p}\int_{\R^N}\frac{x}{|x|} |u(x)|^p\, dx, \quad u\in E\backslash\{0\}.\]
This function is continuous on $E\backslash\{0\}$ and uniformly continuous on the bounded subsets of 
$E\backslash\{u\in E\, :\, \|u\|_{L^p}<r\}$ for every $r>0$. Moreover, $|\beta(u)|<1$ for every $u\neq 0$. 
For $b\in B_1(0)\subset\R^N$ we set
\[ I_b:=\inf_{\substack{u\in\mM\\ \beta(u)=b}}J(u)=\inf_{\substack{v\in\mS^+\\ \beta(\m(v))=b}}\Psi(v) \geq c.\]
We claim that if $c=I_b$ for some $|b|<1$, then $J$ has a nontrivial critical point, i.e., \eqref{eqn:nls_f}
has a nontrivial solution.

Indeed, let $(v_k)_k\subset\mS^+$ with $\beta(\m(v_k))=b$ for all $k\in\N$ be a minimizing sequence for $I_b$, i.e.,
$\llim_{k\to\infty}\Psi(v_k)=I_b=c$. For each $k\in\N$, choose $\delta_k>0$ such that $|\beta(\m(v))-\beta(\m(v_k))|<\frac{1-|b|}{2}$ 
holds for every $v\in\mS^+$ with $\|v-v_k\|\leq 2\delta_k$.
According to Ekeland's variational principle 
(see \cite[Theorem 8.5]{WILLEM}), we can find some $w_k\in\mS^+$ satisfying $c=I_b\leq\Psi(w_k)\leq I_b+\frac2k$, $\|w_k-v_k\|\leq 2\delta_k$
and $\|\Psi'(w_k)\|_\ast\leq \frac8k$.
Setting $u_k:=\m(w_k)$ for all $k\in\N$, we obtain a Palais-Smale sequence $(u_k)_k\subset\mM$ for $J$ at level $I_b=c$ with the additional
property that $|\beta(u_k)|\leq\frac{1+|b|}{2}<1$ for all $k\in\N$. Remark that by Lemma \ref{lem:coercivity} $(u_k)_k$ is a bounded
sequence. Hence, the estimates \eqref{eqn:estim_f_eps}, together with the fact that $u_k\in\mM$ for all $k$, imply 
$\infl_{k\in\N}\|u_k\|_{L^p}>0$.

Suppose by contradiction that there is no $\overline{u}\in E\backslash\{0\}$ such that $J'(\overline{u})=0$. According to Lemma 
\ref{lem:Benci-Cerami} and the assumption $c=c_\infty$, we can find a sequence $(x_k)_k\subset\R^N$ such that $|x_k|\to\infty$ and 
$\|u_k-(x_k\ast u_\infty)\|\to 0$, as $k\to\infty$.
Noticing further that $\|x_k\ast u_\infty\|_{L^p}=\|u_\infty\|_{L^p}>0$ holds for all $k$, and $|\beta(x_k\ast u_\infty)|\to 1$ as $k\to\infty$, the
uniform continuity of $\beta$ gives
\[1=\lim_{k\to\infty}|\beta(x_k\ast u_\infty)|\leq \limsup_{k\to\infty}|\beta(u_k)|\leq\frac{1+|b|}{2}.\]
This contradicts our assumption $|b|<1$, and shows that there must exist some $\overline{u}\in E\backslash\{0\}$ such $J'(\overline{u})=0$.
From Lemma \ref{lem:Benci-Cerami}, it follows that $\llim_{k\to\infty}\|\m(w_k)-\overline{u}\|=0$, and thus, $J(\overline{u})=c$ holds with
$\beta(\overline{u})=b$. This proves the claim.

It remains to see what happens when $c=c_\infty<I_b$ holds for every $|b|<1$.
For $R>0$, let $y=(0,\ldots,0,R)\in\R^N$ and consider the open ball
\[\Omega_R:= B_{\frac43 R}(\textstyle{\frac{y}{3}})=\left\{ (1-s)y + sz\in\R^N\, :\, 0\leq s< 1, \, z\in\partial\Omega_R\right\}.\]
It has the following properties.
\begin{itemize}
  \item[(i)] $0, y\in\Omega_R$, $\frac23R \leq |y-z|\leq 2\min\{|y|,|z|\}=2R$ for all $z\in\partial\Omega_R$, and $|y-z|=2R$ if and 
  only if $z=-y$.

  \item[(ii)] For every $x\in\overline{\Omega}_R\backslash\{y\}$ there exists a unique $(s,z)\in(0,1]\times\partial\Omega_R$ satisfying
  $x=(1-s)y+sz$.

  \item[(iii)] For every $x\in\overline{\Omega}_R\backslash\{0\}$ there exists exactly one $(\tau,\zeta)\in(0,1]\times\partial\Omega_R$
  satisfying $x=\tau \zeta$. Moreover, $\tau$ is given by
  \begin{equation}\label{eqn:sigma}
    \tau(x)=\frac{1}{5R}\left[ \sqrt{15|x|^2+ x_N^2}- x_N\right],
  \end{equation}
  where $x=(x_1,\ldots, ,x_N)\in\ol{\Omega}_R\backslash\{0\}\subset\R^N$.
\end{itemize}
The function $g$: $\overline{\Omega}_R$ $\to$ $\overline{B_1(0)}$ given by
\[g(x)=\left\{\begin{array}{cc}\ds\tau(x)\frac{x}{|x|} & \text{if } x\neq 0 \\ 0 & \text{if }x=0\end{array}\right.\]
is a continuous bijection which satisfies 
$g(\partial \Omega_R)=\partial B_1(0)$. Furthermore, $g$ is smooth on $\Omega_R\backslash\{0\}$.
For $b=(0,\ldots,0,|b|)$ with $0<|b|<1$, we have $g(\frac{5R}{3} b)=b$ and 
$g'(\frac{5R}{3} b)=\frac{3}{5R}  id$. Thus $\text{deg}(g,\Omega_R,b)=1$.

We now define a min-max value as follows. Let $R\geq S_2$ where $S_2$ is given in Lemma \ref{lem:energy}, and consider 
$\gamma_0$: $\partial\Omega_R$ $\to$ $\mM$ given by
\[\gamma_0(z):=\hat\m(z\ast u_\infty)\quad \text{ for all }z\in\partial\Omega_R.\]
We set $\ds\Gamma_R:=\{ \gamma\,:\, \overline{\Omega}_R \to \mM\, :\, \gamma\text{ continous and } \gamma|_{\partial\Omega_R}=\gamma_0\}$ and
\begin{equation}
  c_0:=\inf_{\gamma\in\Gamma_R}\max_{x\in\overline{\Omega}_R}J(\gamma(x)).
\end{equation}
We claim that for $b=(0,\ldots,0,|b|)$ with $0<|b|<1$ fixed, there holds $I_b\leq c_0< 2c_\infty$ for $R$ large enough.

To show the left-hand inequality, consider for each $\gamma\in\Gamma_R$ the homotopy 
$\eta$: $[0,1]\times\overline{\Omega}_R$ $\to$ $\overline{B_1(0)}$ given by
$\eta(\xi, x)= \xi \beta(\gamma(x)) + (1-\xi) g(x)$,  $0\leq \xi\leq 1$, $x\in\overline{\Omega}_R$.
Since $\gamma|_{\partial\Omega_R}=\gamma_0$ and $\xi\beta(\gamma_0(z))+(1-\xi) g(z)\to \frac{z}{|z|}$ uniformly
for $z\in\partial\Omega_R$ and $0\leq \xi\leq 1$, as $R\to\infty$, we obtain $b\notin \eta([0,1]\times\partial\Omega_R)$
for $R$ large enough.
The homotopy invariance of the degree then implies $\text{deg}(\beta\circ \gamma,\Omega_R,b)=\text{deg}(g,\Omega_R,b)=1$.
Using the existence property, we can therefore find some $x_b\in\Omega_R$ for which $\beta(\gamma(x_b))=b$, 
and this gives $I_b\leq J(\gamma(x_b))$. Since $\gamma\in\Gamma_R$ was arbitrarily chosen, we obtain $I_b\leq c_0$.
Lemma \ref{lem:energy} gives the second inequality, when we consider $\gamma_2\in\Gamma_R$ given by
\[\gamma_2((1-s)y + sz)= \hat\m((1-s)(y\ast u_\infty) + s(z\ast u_\infty))\qquad s\in[0,1],\; z\in\partial\Omega_R.\]
In particular, the min-max. level $c_0$ satisfies
\begin{equation}\label{eqn:crit_level}
  c=c_\infty< c_0 <2c_\infty
\end{equation}
for $R$ large enough.

We now wish to prove that $J$ has a (nontrivial) critical point at level $c_0$. For this, we note that
$\mS^+:=\{u\in E^+\, :\, \|u\|=1\}$ is a complete connected $C^1$-Finsler manifold, as a closed connected $C^1$-submanifold
of the Banach space $E^+$. Moreover, for $R\geq S_2$, the 
the family $\mF_R=\{(\m^{-1}\circ \gamma)(\overline{\Omega}_R)\subset\mS^+\, :\, \gamma\in\Gamma_R\}$ of compact subsets of $\mS^+$
is a homotopy-stable family with boundary $B_R:=(\m^{-1}\circ \gamma_0)(\partial\Omega_R)\subset\mS^+$, in the sense
of Ghoussoub \cite[Definition 3.1]{GHOU}.
Since $J(\gamma_0(z))$ converges to $c_\infty$ as $R\to\infty$, uniformly for $z\in\partial\Omega_R$, we have furthermore
\[\sup\Psi(B_R)= \max_{z\in\partial\Omega_R}J(\gamma_0(z))< c_0=\inf_{\gamma\in\Gamma_R}\max_{x\in\overline{\Omega}_R}J(\gamma(x))
  =\inf_{A\in\mF_R}\sup_{v\in A}\Psi(v)\]
for large $R$. Using the min-max. principle \cite[Theorem 3.2]{GHOU}, we can find a sequence 
$(v_k)_k\subset\mS^+$ such that 
$\Psi(v_k)\to c_0$ and $\|\Psi'(v_k)\|_{\ast} \to 0$ as $k\to\infty$. Consequently, the sequence $(\m(v_k))_k\subset\mM$ is a Palais-Smale
sequence for $J$ at level $c_0$. Now, any sign-changing critical point $w$ of $J_\infty$ satisfies $J_\infty(w)>2 c_\infty$ 
(see e.g. \cite[Lemma 2.4]{AckWeth05}). Hence, the estimate \eqref{eqn:crit_level}, the uniqueness of the positive solution of \eqref{eqn:nls_f_inf}
together with Lemma \ref{lem:Benci-Cerami} imply that, up to a subsequence, $\m(v_k)\to \overline{u}$ as $k\to\infty$ for some 
$\overline{u}\in E\backslash\{0\}$ which satisfies $J'(\overline{u})=0$ and $J(\overline{u})=c_0$. This concludes the proof.
$\qed$

\section{Existence of a ground-state solution}
This section is devoted to the existence of a ground-state solution of \eqref{eqn:nls_f} under the conditions of Theorem \ref{thm:gs}.
Therefore, we assume from now on, that $a$ and $f$ satisfy (A1), (F1)--(F4). If $E^0\neq\{0\}$ we suppose, in addition, that 
\eqref{eqn:theta_delta} holds. The proof of Theorem \ref{thm:gs} relies upon Proposition \ref{prop:c<cinf_f} and a similar energy estimate as before
(compare Lemma \ref{lem:S2} and Lemma \ref{lem:h_t_cutoff_f_zero}). This time we shall consider the translate of one ground-state of the limit equation
\eqref{eqn:nls_f_inf} together with a cutoff argument. The latter is well-suited to our setting, since through \eqref{eqn:A2_gs_1} and \eqref{eqn:A2_F_gs_2}
we only control the behavior of $a$ and $F$ respectively, for large $|x|$. We do not have (nor do we require) any information about what happens elsewhere.

For the remainder of this section, we choose some ground-state solution $u_\infty$ of \eqref{eqn:nls_f_inf}.
Our hypotheses ensure that $u_\infty\in H^1(\R^N)\cap\mC(\R^N)$ satisfies either $u_\infty>0$ or $u_\infty<0$ on $\R^N$.
Furthermore, for every $\eps>0$ there exists $C_\eps>0$ such that
\begin{equation}\label{eqn:decay_u_infty_eps}
  |u_\infty(x)|\leq C_\eps e^{-(1-\eps)\sqrt{a_\infty}|x|} \qquad \text{for all }x\in\R^N
\end{equation}
(see \cite{PANKOV08} and \cite[Theorem C.3.5]{SIMON82}). Taking $\theta$ as in \eqref{eqn:theta_delta} if $E^0\neq\{0\}$ and
setting $\theta=0$ otherwise, we fix some $0<\eps<\min\{1-\frac{\alpha(1+\theta)}{(2+\theta)},1-\sqrt{\frac{\alpha}{2}}\}$ and set $A_\eps:=(1-\eps)\sqrt{a_\infty}$.
Moreover, we consider a cut-off function $\chi\in\mC^\infty(\R^N)$, $0\leq \chi\leq 1$, such that 
$\chi(x)=1$ if $|x|\leq 1-\eps$ and $\chi(x)=0$ if $|x|\geq 1$. 

For $R>0$, we set
\[ u_\infty^R(x):=\chi(\frac{x}{R})u_\infty(x), \quad x\in\R^N.\]
Similar to \cite[Lemma 2]{CW04}, we have the following estimates as $R\to\infty$.
\begin{align}
  \int_{\R^N}\left| |\nabla u_\infty|^2 - |\nabla u_\infty^R|^2\right|\, dx &= O( e^{-2(1-\eps)A_\eps R}) \label{eqn:CW_1_f}\\
  \int_{|x|\geq R}|u_\infty|^s\, dx &= O(e^{-s A_\eps R}) \text{ for }s>0.\label{eqn:CW_2_f}
\end{align}
In the following it will be helpful to consider 
\begin{equation}\label{eq:18}
  R_y:=\gamma|y| \qquad \text{for $y \in \R^N \setminus \{0\}$ with some fixed $\gamma \in (0,1)$}. 
\end{equation}
We shall estimate the asymptotic behavior of some integrals involving the eigenfunctions and translates of the ground-state $u_\infty$.
\begin{lem}\label{lem:exp_cutoff_f_zero}
  For $\sigma, \nu>0$ and $0<\beta<\min\{\nu, \nu(1-\gamma)+\sigma\gamma\}$, 
  there exists $C>0$ such that
  \[\int_{\R^N}|y\ast u_\infty^{R_y}|^\sigma e^{-\nu A_\eps |x|}\, dx \leq C e^{-\beta A_\eps |y|}\]
  holds for all $y\in\R^N$. Moreover, if $\sigma>\nu$, the conclusion also holds when $\beta=\nu$.
\end{lem}

\begin{proof}
Let $0<\beta<\min\{\nu, \nu(1-\gamma)+\sigma\gamma\}$ and consider $y\in\R^N$. There holds
\[ e^{\beta A_\eps |y|}\int_{\R^N}|y\ast u_\infty^{R_y}|^\sigma e^{-\nu A_\eps|x|}\, dx 
  \leq C \int_{\{|x|\leq\gamma|y|\}} e^{(\beta-\sigma)A_\eps|x|} e^{-(\nu-\beta)A_\eps|x+y|}\, dx,\]
using \eqref{eqn:decay_u_infty_eps}. Now, if $\beta\geq\sigma$, we can write
\begin{align*}
\int_{\{|x|\leq\gamma|y|\}} e^{(\beta-\sigma)A_\eps|x|} e^{-(\nu-\beta)A_\eps|x+y|}\, dx
  &\leq C |y|^N e^{(\beta-\sigma)A_\eps\gamma|y|} e^{-(\nu-\beta)A_\eps(1-\gamma)|y|}\\
  &= C |y|^N e^{-(\nu(1-\gamma)+\sigma\gamma-\beta)A_\eps|y|}.
\end{align*}
The assumption $\beta<\nu(1-\gamma)+\sigma\gamma$ then gives the desired result.

On the other hand, if $\beta<\sigma$, we have
\[\int_{\{|x|\leq\gamma|y|\}} e^{(\beta-\sigma)A_\eps|x|} e^{-(\nu-\beta)A_\eps|x+y|}\, dx
  \leq e^{-(\nu-\beta)A_\eps(1-\gamma)|y|}\int_{\R^N}e^{-(\sigma-\beta)A_\eps|x|}\, dx,\]
and the conclusion follows.
\end{proof}

As a consequence of the preceding lemma, we obtain the following estimates for the integrals below.
\begin{equation}\label{eqn:exp_decay_cutoff_1}
  \int_{\R^N} |y\ast u^{R_y}_\infty|\, |e_i| \, dx \leq C e^{-A_\eps|y|} \quad\text{for all }y\in\R^N, 1\leq i\leq n+l,
\end{equation}
with some constant $C=C(\eps,\gamma)>0$.
Moreover, for any $0<\delta<(\sqrt{1+\frac{|\lambda_i|}{a_\infty}}-1)(1-\gamma)$, 
there exists $C=C(\eps,\gamma,\delta)>0$ satisfying
\begin{equation}\label{eqn:exp_decay_cutoff_2}
  \int_{\R^N} (y\ast u^{R_y}_\infty)\, |e_i| \, dx \leq C e^{-(1+\delta)A_\eps|y|} \quad\text{for all }y\in\R^N, 1\leq i\leq n.
\end{equation}

\begin{lem}\label{lem:h_t_cutoff_f_zero}
  \begin{itemize}
    \item[(i)] There exists $S_3>0$ such that $y\ast u^{R_y}_\infty\notin H^-\oplus H^0$ for all $|y|\geq S_3$.
    
    \item[(ii)] For $|y|\geq S_3$, let $t_y>0$, $h_y^1\in H^-$ and $h_y^2\in H^0$ satisfy 
    \[\hat\m(y\ast u^{R_y}_\infty)=t_y[y\ast u^{R_y}_\infty + h_y^1+h_y^2].\] Then, as $|y|\to\infty$, there holds
    $t_y\to 1$  and $\|h_y^1\|+\|h_y^2\|\to 0$.
    More precisely, there exists $C=C(\eps,\gamma)>0$ such that
    \begin{equation}\label{eqn:h1_cutoff_f_zero}
      \|h_y^1\| \leq C e^{-\frac{(2+\theta)}{2(1+\theta)}A_\eps|y|}
    \end{equation}
    and
    \begin{equation}\label{eqn:h2_cutoff_f_zero}
      \|h_y^2\|\leq C e^{-\frac{1}{1+\theta}A_\eps|y|}
    \end{equation}
    hold for all $|y|\geq S_3$.
  \end{itemize}
\end{lem}

\begin{proof}
\begin{itemize}
  \item[(i)] Since $a(x)\to a_\infty$ as $|x|\to\infty$, and since $|x|\leq R_y$ implies $|x+y|\geq (1-\gamma)|y|$,
  the dominated convergence theorem, together with \eqref{eqn:CW_1_f} and \eqref{eqn:CW_2_f}, gives
  \begin{align}
    & \int_{\R^N} |\nabla(y\ast u_\infty^{R_y})|^2 + a(x) |y\ast u_\infty^{R_y}|^2\, dx \nonumber 
    =\int_{\R^N} |\nabla u_\infty^{R_y}|^2 + a(x+y) | u_\infty^{R_y} |^2\, dx \\
    &\qquad \lra \int_{\R^N}|\nabla u_\infty|^2+a_\infty|u_\infty|^2\, dx \quad \text{as }|y|\to\infty. \label{eqn:A_cutoff_f_zero}
  \end{align}
  Since $\int_{\R^N}|\nabla h|^2+a(x)|h|^2\, dx =-\|h\|^2\leq 0$ for all $h\in H^-\oplus H^0$, the conclusion follows
  from \eqref{eqn:A_cutoff_f_zero} and the fact that $0<2c_\infty\leq \int_{\R^N}|\nabla u_\infty|^2+a_\infty|u_\infty|^2\, dx$.

  \item[(ii)] Since $J'(t_y [y\ast u_\infty^{R_y}+ h^1_y+h^2_y])(t_y [y\ast u_\infty^{R_y}+ h^1_y+h^2_y])=0$ holds,
  we find
  \begin{align*}
    0< 2c &\leq 2J(t_y [y\ast u_\infty^{R_y}+ h^1_y+h^2_y])+2\int_{\R^N}F(x,t_y [y\ast u_\infty^{R_y}+ h^1_y+h^2_y])\, dx\\
    &= t_y^2 \int_{\R^N} |\nabla[y\ast u_\infty^{R_y}+h^1_y]|^2 + a(x) |y\ast u_\infty^{R_y}+h^1_y|^2\, dx\\
    &\leq t_y^2 \Bigl\{ \int_{\R^N}|\nabla y\ast u_\infty^{R_y}|^2+a(x)|y\ast u_\infty^{R_y}|^2\, dx 
    +2C\|u_\infty\|\, \|h^1_y\| -\|h^1_y\|^2\Bigr\}.
  \end{align*}
  From \eqref{eqn:A_cutoff_f_zero}, we deduce that
  \begin{equation}\label{eqn:h_bd_t_pos_zero}
    \sup_{|y|\geq S_3} \|h^1_y\|<+\infty\quad \text{ and } \quad \inf_{|y|\geq S_3}t_y>0.
  \end{equation}
  Furthermore, we claim that $\{t_y(y\ast u_\infty^{R_y}+h^1_y+h^2_y)\,:\,|y|\geq S_3\}$ is bounded, so that
  \begin{equation}\label{eqn:t_bd_zero}
    \sup_{|y|\geq S_3} t_y <+\infty\quad \text{ and }\quad \sup_{|y|\geq S_3} \|h^2_y\|<+\infty
  \end{equation}
  holds. Indeed, suppose by contradiction that (up to a subsequence) $\|t_y(y\ast u_\infty^{R_y}+h^1_y+h^2_y)\|\to\infty$
  as $|y|\to\infty$. Setting $w_y:=\frac{y\ast u_\infty^{R_y}+h^1_y+h^2_y}{\|y\ast u_\infty^{R_y}+h^1_y+h^2_y\|}
  =s_y(y\ast u_\infty^{R_y})+z_y^1+z_y^2$,
  we first note that $s_y$ stays bounded as $|y|\to\infty$, since $\|s_y(y\ast u_\infty^{R_y})^+\|\leq 1=\|w_y\|$
 and $\|(y\ast u_\infty^{R_y})^+\|\not\to 0$ as $|y|\to\infty$, according to \eqref{eqn:A_cutoff_f_zero}. Hence, $\dim(H^-\oplus H^0)<\infty$ implies
 that, up to a subsequence, $s_y\to s$, $z_y^1\to z_1$ and $z_y^2\to z_2$ as $|y|\to\infty$
  for some $s\geq 0$, $z_1\in H^-$ and $z_2\in H^0$.
  
  If $s\neq 0$, we obtain that
  \[ \|(-y)\ast w_y - su_\infty^{R_y} -(-y)\ast z_1 -(-y)\ast z_2\|\to 0\quad\text{ as }|y|\to\infty.\]
  Hence, $(-y)\ast w_y \rightharpoonup su_\infty\neq 0$ as $|y|\to\infty$, which implies that (up to a subsequence)
  $w_y(x+y)\to su_\infty(x)$ a.e. on $\R^N$, and since $w_y=\frac{t_y(y\ast u_\infty^{R_y}
  +h^1_y+h^2_y)}{\|t_y(y\ast u_\infty^{R_y}+h^1_y+h^2_y)\|}$ and $|u_\infty|>0$, 
  we obtain that for a.e. $x\in\R^N$, $|t_y[u_\infty^{R_y}(x)+h^1_y(x+y)+h^2_y(x+y)]|\to\infty$ as $|y|\to\infty$.
  But then, Fatou's Lemma gives
  \begin{align*}
    &\int_{\R^N}\frac{ |\nabla[y\ast u_\infty^{R_y}+h^1_y]|^2 + a(x) |y\ast u_\infty^{R_y}+h^1_y|^2}{
     \|[y\ast u_\infty^{R_y}+h^1_y+h^2_y]\|^2}\, dx \\
    &= \int_{\R^N}\frac{f(x,t_y[y\ast u_\infty^{R_y}+h^1_y+h^2_y])t_y[y\ast u_\infty^{R_y}+h^1_y+h^2_y]}{
     \|t_y[y\ast u_\infty^{R_y}+h^1_y+h^2_y]\|^2}\, dx \\
    &\geq 2\int_{\R^N}\frac{F(x+y,t_y[u_\infty^{R_y}(x)+h^1_y(x+y)+h^2_y(x+y)])}{\{t_y[u_\infty^{R_y}(x)
     +h^1_y(x+y)+h^2_y(x+y)]\}^2}(w_y(x+y))^2\, dx\\
    & \to\infty \quad \text{as }|y|\to\infty,   
  \end{align*}
  which contradicts the boundedness of $y\ast u_\infty^{R_y}+h_y^1$ and the fact that $\|(y\ast u_\infty^{R_y})^+\|\not\to 0$ as $|y|\to\infty$.
  On the other hand, if $s=0$, then $w_y\to z_1+z_2$ strongly in $H$, and hence $\|z_1+z_2\|=1\neq 0$ holds. The same argument
  as above gives (up to a subsequence) $|t_y[u_\infty^{R_y}(x-y)+h^1_y(x)+h^2_y(x)]|\to\infty$ as $|y|\to\infty$
  and, using Fatou's Lemma, we obtain again a contradiction. Therefore, the claim is proved and \eqref{eqn:t_bd_zero} holds.

  Now $J'(t_y[y\ast u_\infty^{R_y} +h^1_y +h^2_y])(h^1_y+h^2_y)=0$ implies
  \begin{equation}\label{eqn:estim_h1_h2_cutoff}
  \begin{split}
    0 &\leq \|h^1_y\|^2 + \int_{\R^N}\frac{f(x,t_y[y\ast u_\infty^{R_y} +h^1_y +h^2_y])}{t_y[y\ast u_\infty^{R_y} 
    +h^1_y +h^2_y]}(h^1_y+h^2_y)^2\, dx\\
    &=\int_{\R^N}\nabla(y\ast u_\infty^{R_y})\cdot\nabla h_y^1 + a(x)(y\ast u_\infty^{R_y})h_y^1\, dx\\
    & \quad\quad - \int_{\R^N}\frac{f(x,t_y[y\ast u_\infty^{R_y} +h^1_y +h^2_y])}{t_y[y\ast u_\infty^{R_y} +h^1_y 
     +h^2_y]}(y\ast u_\infty^{R_y})(h^1_y+h^2_y)\, dx.
  \end{split}
  \end{equation}
  Writing $h^1_y=\suml_{i=1}^n A_i^y e_i$ and $h^2_y=\suml_{i=n+1}^{n+l} A_i^y e_i$ and using
  the fact that $t_y[y\ast u_\infty^{R_y} +h^1_y +h^2_y]$ is bounded in $L^\infty(\R^N)$, we obtain
  \begin{equation}\label{eqn:h_tilde_y_f_zero}
   \begin{split}
     &\|h^1_y\|^2 + \int_{\R^N}\frac{f(x,t_y[y\ast u_\infty^{R_y} +h^1_y +h^2_y])}{t_y[y\ast u_\infty^{R_y} 
      +h^1_y +h^2_y]}(h^1_y+h^2_y)^2\, dx \\
     &\leq C \sum_{i=1}^{n+l} |A_i^y| \int_{\R^N} |y\ast u_\infty^{R_y}|\, |e_i|\, dx
     \leq \tilde C\, \|h^1_y  +h_y^2\| e^{-A_\eps|y|},
   \end{split}
  \end{equation}
  for some $\tilde C>0$, by \eqref{eqn:exp_decay_cutoff_1}, \eqref{eqn:h_bd_t_pos_zero}, \eqref{eqn:t_bd_zero} and 
  since all norms are equivalent on $H^-\oplus H^0$.
  In the case $E^0=\{0\}$, the lemma is proved, since $h^2_y=0$ holds.
  
  If $E^0\neq\{0\}$, we notice that (F4) and \eqref{eqn:theta_delta} imply 
  \[\inf\{\frac{|f(x,u)|}{|u|^{1+\theta}}\,:\,x\in\R^N,\, 0<|u|\leq r\}>0\]
  for all $r>0$. Choosing $r=\supl_{|y|\geq S_3}\|t_y[y\ast u_\infty^{R_y} +h^1_y +h^2_y]\|_\infty<\infty$, we obtain
  \begin{align*}
    \int_{\R^N}|h^1_y +h^2_y|^{2+\theta}\, dx &\leq \frac{\max\{1,2^{\theta-1}\}}{(\infl_{|y|\geq S_3}t_y)^\theta} \int_{\R^N}|t_y[y\ast u_\infty^{R_y} 
		 +h^1_y +h^2_y]|^\theta(h_y^1+h_y^2)^2\, dx\\
    &\quad\quad + \max\{1,2^{\theta-1}\}\int_{\R^N}|y\ast u_\infty^{R_y}|^\theta(h_y^1+h_y^2)^2\, dx\\
    &\leq C \int_{\R^N}\frac{f(x,t_y[y\ast u_\infty^{R_y} +h^1_y +h^2_y])}{t_y[y\ast u_\infty^{R_y} 
	+h^1_y +h^2_y]}(h^1_y+h^2_y)^2\, dx\\
    &\quad\quad +\max\{1,2^{\theta-1}\}\int_{\R^N}|y\ast u_\infty^{R_y}|^\theta(h_y^1+h_y^2)^2\, dx\\
    &\leq \tilde C \|h^1_y +h_y^2\| e^{-A_\eps|y|} + C' \|h^1_y+h_y^2\|^2 e^{-\frac{\theta}{1+\theta} A_\eps|y|}
  \end{align*}
  since $0<\frac{\theta}{1+\theta}<\min\{2, 2(1-\gamma)+\theta\gamma\}$, using Lemma \ref{lem:exp_cutoff_f_zero} and \eqref{eqn:h_tilde_y_f_zero}.

  Since all norms are equivalent on $E^-\oplus E^0$, the preceding estimate gives
  \begin{equation}\label{eqn:h1_f_zero}
    \|h_y^1+h_y^2\|\leq C e^{-\frac{1}{1+\theta}A_\eps|y|}\quad \text{for all }|y|\geq S_3.
  \end{equation}
  Combining this estimate with \eqref{eqn:h_tilde_y_f_zero}, we find
  \begin{equation}\label{eqn:h2_f_zero}
    \|h_y^1\|\leq C e^{-\frac{(2+\theta)}{2(1+\theta)}A_\eps|y|}\quad \text{for all }|y|\geq S_3.
  \end{equation}
  Hence, $\|h^1_y+h_y^2\|\to 0$ as $|y|\to\infty$ which, in turn, implies that
  \[\int_{\R^N} |\nabla[y\ast u_\infty^{R_y}+h^1_y]|^2 + a(x) [y\ast u_\infty^{R_y}+h^1_y]^2\, dx
    \to \int_{\R^N}|\nabla u_\infty|^2+a_\infty u_\infty^2\, dx,\]
  as $|y|\to\infty$. From \eqref{eqn:h_bd_t_pos_zero} and \eqref{eqn:t_bd_zero}, we can assume that (up to a subsequence) $t_y\to T>0$ as $|y|\to\infty$. 
  Then, we find
  \[0 = J'(t_y[y\ast u_\infty^{R_y}+h^1_y+h^2_y])t_y[y\ast u_\infty^{R_y}+h^1_y+h^2_y] \to J_\infty'(Tu_\infty)Tu_\infty,\]
  as $|y|\to\infty$, i.e., $T u_\infty\in\mM_\infty$. Since $u_\infty\in\mM_\infty$, it follows that $T=1$. This concludes the proof.
\end{itemize}
\end{proof}

\begin{remq}\label{remq:nonresonant}
 In the case $E^0=\{0\}$, we have $h_y^2=0$, and using \eqref{eqn:exp_decay_cutoff_2} instead of \eqref{eqn:exp_decay_cutoff_1} in
 \eqref{eqn:h_tilde_y_f_zero} in the proof above, we obtain a better estimate for the decay of $h^1_y$. Namely,
 for every $0<\delta<\min\{(\sqrt{1+\frac{|\lambda_i|}{a_\infty}}-1)(1-\gamma)\, :\, 1\leq i\leq n\}$,
 \begin{equation}\label{eqn:h1_cutoff_f}
   \|h_y^1\| \leq C e^{-(1+\delta)A_\eps|y|}
 \end{equation}
 holds for all $|y|\geq S_3$ with some $C=C(\eps,\gamma,\delta)>0$.
\end{remq}
We are now ready to prove Theorem \ref{thm:gs}. Let therefore conditions (A1), (F1)--(F4) and, if $\ker(-\Delta+a)=E^0\neq\{0\}$ holds,
\eqref{eqn:theta_delta} be satisfied.
\begin{proof}[Proof of Theorem \ref{thm:gs}]
As before, let
$0<\eps<\min\{1-\frac{\alpha(1+\theta)}{(2+\theta)},1-\sqrt{\frac{\alpha}{2}}\}$. In
case \eqref{eqn:F_gs_1}, we may also assume
\begin{equation}\label{eqn:remq_mu}
  \alpha<\mu<2 \qquad \text{and}\qquad \eps< \min\{1-\frac{\alpha}{\mu},1-\frac{\mu}{2}\}.  
\end{equation}
We fix $\gamma>0$ such that 
$\frac{\alpha}{2(1-\eps)^2}<\gamma<1$ and consider the
corresponding radii $R_y$ as defined in \eqref{eq:18}. We claim that
we can find $S\geq S_3$ such that
\begin{equation}\label{eqn:estim_J_cutoff}
  J(\hat\m(y\ast u_\infty^{R_y})) < c_\infty \qquad \text{for all $|y|\geq S$,}
\end{equation}
where $S_3$ is given by Lemma \ref{lem:h_t_cutoff_f_zero}.
The conclusion of the theorem then follows from Proposition \ref{prop:c<cinf_f}. Let us prove \eqref{eqn:estim_J_cutoff}. 
With the notation of  Lemma \ref{lem:h_t_cutoff_f_zero}, consider $|y|\geq S_3$ and $(t_y,h^1_y,h^2_y)\in (0,\infty)\times E^-\times E^0$ with
$\hat\m (y\ast u_\infty^{R_y})=t_y[y\ast u_\infty^{R_y}+h^1_y+h^2_y]$. There holds
\begin{align*}
  &J(\hat\m(y\ast u_\infty^{R_y})) = \frac{t_y^2}{2}\int_{\R^N} |\nabla u_\infty^{R_y}|^2+a_\infty|u_\infty^{R_y}|^2 dx 
  -\int_{\R^N} F_\infty(t_y u_\infty^{R_y})\, dx \\
  &\quad  +\frac{t_y^2}{2}\Bigl[ \int_{\R^N} |\nabla h^1_y|^2 + a(x) |h^1_y|^2\, dx
   + 2\int_{\R^N} \nabla(y\ast u_\infty^{R_y})\cdot\nabla h^1_y +  a(x) (y\ast u_\infty^{R_y})h^1_y\, dx \\
  &\quad + \int_{\R^N} (a(x+y)-a_\infty) |u_\infty^{R_y}|^2 dx\Bigr]
  - \Bigl[ \int_{\R^N} F(x+y,t_y u_\infty^{R_y})-F_\infty(t_y u_\infty^{R_y})\, dx \\
  &\quad\quad + \int_{\R^N} \{F(x,t_y[y\ast u_\infty^{R_y}+h^1_y+h_y^2])-F(x,t_y(y\ast u_\infty^{R_y}))\}\, dx  \Bigr].
\end{align*}    
In the following, we let $K_1,K_2,\dots$ denote positive constants
depending possibly on $S_3$, $\alpha$, $\theta$, $\eps$, $\gamma$, $\mu$ and $u_\infty$
but not on $y$. From \eqref{eqn:CW_1_f}, \eqref{eqn:CW_2_f} and the fact that $t_y$ remains bounded
as $|y|\to\infty$, it follows that
\begin{align*}
  &\frac{t_y^2}{2}\int_{\R^N} (|\nabla u_\infty^{R_y}|^2+a_\infty|u_\infty^{R_y}|^2)\, dx 
  - \int_{\R^N} F_\infty(t_y u_\infty^{R_y})\, dx\\
  &\leq J_\infty(t_y u_\infty) + \frac{t_y^2}{2}\int_{\R^N} \bigl(|\nabla u_\infty^{R_y}|^2-|\nabla u_\infty|^2\bigr)\, dx
  + \int_{\{|x|\geq (1-\eps)R_y\}} F_\infty(t_y u_\infty)\, dx\\
  &\leq c_\infty + K_1\, e^{-2(1-\eps)A_\eps R_y}
\end{align*}
for $y \in \R^N \setminus \{0\}$, using the fact that $J_\infty(t_y u_\infty)\leq J_\infty(u_\infty)=c_\infty$,
since $u_\infty\in\mM_\infty$ holds.

We now consider the case where assumption (a) of Theorem \ref{thm:gs}
is satisfied. Then \eqref{eqn:A2_gs_1} implies
\begin{align*}
  &\int_{\R^N} (a(x+y)-a_\infty) |u_\infty^{R_y}|^2\, dx
  =\int_{\{|x|\leq R_y\}} \!\!(a(x+y)-a_\infty) |u_\infty^{R_y}|^2\, dx\\
  &\qquad \leq -C_1 e^{-\alpha\sqrt{a_\infty} |y|}\!\int_{\{|x|\leq 1\}}\!\!\!e^{-\alpha\sqrt{a_\infty} |x|} u_\infty^2\, dx 
  \leq -K_2\, e^{-\alpha\sqrt{a_\infty}|y|}
\end{align*}
for $|y|\geq \max \bigl\{\frac{1}{\gamma},\frac{S_0}{1-\gamma}\bigr\}$. 
Furthermore, \eqref{eqn:F_gs_1} and Lemma \ref{lem:exp_cutoff_f_zero} together with \eqref{eqn:remq_mu} yield
\begin{align*}
  \int_{\R^N}&F(x+y,t_y u_\infty^{R_y})-F_\infty(t_y u_\infty^{R_y})\, dx 
  \geq -C_2 \int_{\R^N}\!\!\!\!e^{-\mu\sqrt{a_\infty} |x+y|} [|t_yu_\infty^{R_y}|^2+|t_yu_\infty^{R_y}|^p]\, dx \\
  &= -C_2 \int_{\R^N}e^{-\mu\sqrt{a_\infty} |x|} [|t_y (-y)*u_\infty^{R_y}|^2+|t_y (-y)*u_\infty^{R_y}|^p]\, dx 
  \geq - K_3\, e^{-\mu A_\eps|y|} 
\end{align*}
for $|y| \ge S_3$. If $E^0\neq\{0\}$, we obtain moreover
\begin{align*}
  &\int_{\R^N} \nabla(y\ast u_\infty^{R_y})\cdot\nabla h^1_y +  a(x) (y\ast u_\infty^{R_y})h^1_y\, dx\\
  &\qquad \leq |\lambda_1|\,\|h^1_y\| \max_{1\leq i\leq n}\int_{\R^N}|y\ast u_\infty^{R_y}|\, |e_i|\, dx
  \leq K_4\, e^{-\frac{(4+3\theta)}{2(1+\theta)}A_\eps|y|},
\end{align*}
and 
\begin{align*}
  &\int_{\R^N} \{F(x,t_y[y\ast u_\infty^{R_y}+h^1_y+h_y^2])-F(x,t_y(y\ast u_\infty^{R_y}))\}\, dx\\
  &\qquad \geq -\int_{\R^N} |f(x,t_y(y\ast u_\infty^{R_y})) t_y(h^1_y+h_y^2)|\, dx\\
  &\qquad \geq -K_5\|h^1_y+h_y^2\| \max_{1\leq i\leq n+l}\int_{\R^N}|y\ast u_\infty^{R_y}|\, |e_i|\, dx 
  \geq - K_6\,e^{-\frac{(2+\theta)}{1+\theta}A_\eps|y|},
\end{align*}
for $|y| \ge S_3$ as a consequence of \eqref{eqn:ineq_F_f}, \eqref{eqn:exp_decay_cutoff_1}, \eqref{eqn:h1_cutoff_f_zero} and \eqref{eqn:h2_cutoff_f_zero}.

On the other hand, when $E^0=\{0\}$, we fix 
\[0<\delta<\min\Bigl\{(\sqrt{1+\frac{|\lambda_i|}{a_\infty}}-1)(1-\gamma)\,:\, 1\leq i\leq n \Bigr\}\]
and use Remark \ref{remq:nonresonant} and
\eqref{eqn:exp_decay_cutoff_2} to estimate
\begin{align*}
  &\int_{\R^N} \nabla(y\ast u_\infty^{R_y})\cdot\nabla h^1_y +  a(x) (y\ast u_\infty^{R_y})h^1_y\, dx\\
  &\qquad \leq |\lambda_1|\,\|h^1_y\| \max_{1\leq i\leq n}\int_{\R^N}|y\ast u_\infty^{R_y}|\,|e_i|\, dx
  \leq K_7\, e^{-2(1+\delta)A_\eps|y|},
\end{align*}
and, since $h^2_y=0$ in this case,   
\begin{align*}
  &\int_{\R^N} \{F(x,t_y[y\ast u_\infty^{R_y}+h^1_y+h_y^2])-F(x,t_y(y\ast u_\infty^{R_y}))\}\, dx\\
  &\qquad \geq -\int_{\R^N} |f(x,t_y(y\ast u_\infty^{R_y})) t_y h^1_y|\, dx\\
  &\qquad \geq -K_8\,\|h^1_y\| \max_{1\leq i\leq n}\int_{\R^N} |y\ast u_\infty^{R_y}|\,|e_i|\, dx   \geq - K_9\,e^{-2(1+\delta)A_\eps|y|},
\end{align*}
Putting the previous estimates together and noting in addition that $$\int_{\R^N} |\nabla h^1_y|^2 + a(x) |h^1_y|^2\, dx\le 0,$$
we obtain, for $|y| \ge \max \{S_3,\frac{S_0}{1-\gamma},\frac{1}{\gamma}\}$,
\begin{align*}
  J(\hat\m(y\ast u_\infty^{R_y})) & \leq c_\infty + K_1\, e^{-2(1-\eps)A_\eps R_y}
  -K_2\,  e^{-\alpha\sqrt{a_\infty} |y|}+K_3\, e^{-\mu A_\eps |y|} \\ 
  & \quad + (K_4+K_6) e^{-\frac{(2+\theta)}{1+\theta}A_\eps|y|}
\end{align*}
in the case where $E^0\neq\{0\}$, and
\begin{align*}
  J(\hat\m(y\ast u_\infty^{R_y})) & \leq c_\infty + K_1 e^{-2(1-\eps)A_\eps R_y}
  -K_2\,  e^{-\alpha\sqrt{a_\infty} |y|}+K_3\, e^{-\mu A_\eps |y|} \\
  &\quad  +(K_7+K_9) e^{-2(1+\delta)A_\eps|y|}
\end{align*}
when $E^0=\{0\}$, respectively. Our choice of $\gamma$ implies $2(1-\eps)\gamma A_\eps>\alpha\sqrt{a_\infty}$, and our choice of $\eps$ gives
$\frac{(2+\theta)}{1+\theta}A_\eps>\alpha\sqrt{a_\infty}$ and $\mu A_\eps>\alpha\sqrt{a_\infty}$. Hence, the conclusion follows, since $K_2>0$.

\medskip

\noindent Next we consider the case of assumption (b) in Theorem
\ref{thm:gs}. By Lemma~\ref{lem:h_t_cutoff_f_zero} (ii),
we may choose $\eta>0$ such that 
\[\eta\leq\inf\bigl \{t_y |u_\infty(x)|\, :\, |y|\geq S_3,\: |x|\leq 1 \bigr\} 
  \leq \sup \bigl\{t_y |u_\infty(x)|\, :\, |y|\geq S_3,\: |x|\leq 1\bigr\} \le \frac{1}{\eta},\]
and we consider $C_\eta,S_\eta$ such that \eqref{eqn:A2_F_gs_2} holds
with this choice of $\eta$. Since by assumption $a(x)\leq a_\infty$
for all $|x|\geq S_0$, we obtain 
\[\int_{\R^N} (a(x+y)-a_\infty) |u_\infty^{R_y}|^2\, dx\leq 0 
\qquad \text{for $|y|\geq \frac{S_0}{1-\gamma}$.}\]
Moreover, 
\begin{align*}
  \int_{\R^N}&\!\! F(x+y,t_y u_\infty^{R_y})-F_\infty(t_y u_\infty^{R_y})\, dx \geq C_\eta e^{-\alpha\sqrt{a_\infty} |y|}\int_{\{|x|\leq 1\}} e^{-\alpha\sqrt{a_\infty} |x|}\, dx\\
  &= \tilde{C_\eta} e^{-\alpha\sqrt{a_\infty} |y|} \qquad
  \text{for $|y|\geq \max \bigl\{\frac{1}{\gamma},\frac{S_\eta}{(1-\gamma)}\bigr\}$ with a constant $\tilde{C_\eta}>0$.} 
\end{align*}
Together with the previous estimates, we find, for $|y| \ge \max\{S_3,\frac{S_0}{(1-\gamma)},\frac{S_\eta}{(1-\gamma)},\frac{1}{\gamma}\}$,
\[ J(\hat\m(y\ast u_\infty^{R_y})) \leq c_\infty + K_1\, e^{-2(1-\eps)A_\eps R_y}
-\tilde{C_\eta} e^{-\alpha\sqrt{a_\infty} |y|} +(K_4+K_6) e^{-\frac{(2+\theta)}{1+\theta}A_\eps|y|}\]
in the case where $E^0\neq\{0\}$, and
\[J(\hat\m(y\ast u_\infty^{R_y})) \leq c_\infty + K_1\, e^{-2(1-\eps)A_\eps R_y}
  -\tilde{C_\eta} e^{-\alpha\sqrt{a_\infty} |y|} +(K_7+K_9) e^{-2(1+\delta)A_\eps|y|}\]
when $E^0=\{0\}$, respectively. The conclusion follows as above from
the choice of $\gamma$ and $\eps$.
\end{proof}

\appendix
\section{A nonlinear splitting property}
The aim of this section is to prove the following result, which was
needed in the proof of Lemma \ref{lem:Benci-Cerami}. The proof is an adaptation
of an argument in \cite[Appendix]{AckWeth05} to our setting.  
Throughout this section, we assume (F1)-(F4).
\begin{prop}\label{prop:decomposition}
  Let $(u_k)_k\subset H^1(\R^N)$ and $\bar u\in H^1(\R^N)$ be such that $u_k\rightharpoonup \bar u$ weakly in $H^1(\R^N)$ and
  $|\bar u(x)|\to 0$ as $|x|\to\infty$. Then, as $k\to\infty$,
  \begin{align}
    &\int_{\R^N} F(x,u_k)-F(x,\bar u)-F(x,u_k-\bar u)\, dx \to 0, \\
    &\int_{\R^N}[f(x,u_k)-f(x,\bar u)-f(x,u_k-\bar u)]\varphi\, dx \to 0,
  \end{align}
  uniformly in $\|\varphi\|\leq 1$.
\end{prop}
\begin{proof}
We set $\tilde C:=\sup\limits_k\|u_k\|<+\infty$ and fix some $\eps>0$. According to (F2), we can choose $0<s_0<1$ such that $|f(x,t)|\leq \eps |t|$
and hence $|F(x,t)|\leq \frac{\eps}{2}|t|^2$ holds for all $|t|\leq 2s_0$ and all $x\in\R^N$. Moreover, by assumption, there exists $R>0$ such that $|\bar u(x)|\leq s_0$
for all $|x|\geq R$ and $\int_{\R^N\backslash B_R(0)}|\bar u|^2\, dx<1$.

Hence, we obtain
\begin{align*}
  \int_{\R^N\backslash B_R(0)} |f(x,\bar u)|\, |\varphi|\, dx &\leq \eps (\int_{\R^N\backslash B_R(0)}|\bar u|^2\, dx)^{\frac12}\|\varphi\|_{L^2}<\eps \|\varphi\|_{L^2}\\
  \text{and }\int_{\R^N\backslash B_R(0)} |F(x,\bar u)|\, dx &\leq \frac{\eps}{2} \int_{\R^N\backslash B_R(0)}|\bar u|^2\, dx<\frac{\eps}{2}.
\end{align*}

Setting $\mU_k:=\{x\in\R^N\backslash B_R(0)\, : \,  |u_k(x)|\leq s_0\}$ for $k\in\N$, we find
\begin{align*}
  \int_{\mU_k} |f(x,u_k)-f(x,u_k-\bar u)|\, |\varphi|\, dx &\leq \eps (\|u_k\|_{L^2}+\|u_k-\bar u\|_{L^2})\|\varphi\|_{L^2}\leq 3\tilde{C}\eps \|\varphi\|_{L^2}\\
  \text{and }\int_{\mU_k} |F(x,u_k)-F(x,u_k-\bar u)|\, dx &\leq \frac{\eps}{2} (\|u_k\|_{L^2}^2+\|u_k-\bar u\|_{L^2}^2)\leq \frac{5\eps}{2}\tilde{C}^2.
\end{align*}

Now we consider $\mV_s^k:=\{x\in \R^N\backslash B_R(0)\, :\, |u_k(x)|\geq s\}$ for $k\in\N$ and $s>2$. Notice that in particular $|u_k(x)-\bar u(x)|> s-1>1$
holds for all $x\in\mV_s^k$. As a consequence, we can write
\begin{align*}
  \int_{\mV_s^k}|f(x,u_k)-f(x,u_k-\bar u)|\, |\varphi|\, dx &\leq 2C_0 \int_{\mV_s^k}(|u_k|^{p-1}-|u_k-\bar u|^{p-1}) |\varphi|\, dx\\
  & \leq 2C_0 (s-1)^{p-2^\ast}\int_{\mV_s^k}(|u_k|^{2^\ast-1}-|u_k-\bar u|^{2^\ast-1}) |\varphi|\, dx\\
  & \leq 2C_0(s-1)^{p-2^\ast}\|u_k\|_{L^{2^\ast}}^{2^\ast-1}-\|u_k-\bar u\|_{L^{2^\ast}}^{2^\ast-1}) \|\varphi\|_{L^{2^\ast}}\\
  & \leq \hat C (s-1)^{p-2^\ast} \tilde C^{2^\ast-1}\|\varphi\|.
\end{align*}
Similarly, we find
\[\int_{\mV_s^k}|F(x,u_k)-F(x,u_k-\bar u)|\, dx \leq \bar C (s-1)^{p-2^\ast} \tilde{C}^{2^\ast}\]
for all $k\in\N$. Choosing $s>2$ large enough, we can achieve $\hat C(s-1)^{p-2^\ast}\tilde{C}^{2^\ast -1}<\eps$ and $\bar C (s-1)^{p-2^\ast}\tilde{C}^{2^\ast}<\eps$.

For the next step, we remark that from the continuity of $f_\infty$ and $F_\infty$, we may choose $\delta=\delta(\eps)>0$ such that $|f_\infty(u)-f_\infty(v)|<\frac{s_0\eps}{3}$
and $|F_\infty(u)-f_\infty(v)|<\frac{s_0^2\eps}{3}$ hold for all $|u-v|<\delta$, $|u|, |v|\leq s+1$.
By (A1), we may now pick $R_0=R_0(\eps)\geq R$ satisfying $|\bar u(x)|<\delta$ for all $|x|\geq R_0$ and $|f(x,u)-f_\infty(u)|<\frac{s_0\eps}{3}$ for all $|u|\leq s+1$ and all $|x|\geq R_0$.
Setting $\mW_k:=\{x\in\R^N\backslash B_{R_0}(0)\, :\, s_0\leq |u_k(x)|\leq s\}$ for $k\in\N$, we obtain
\begin{align*}
  &\int_{\mW_k}|f(x,u_k)-f(x,u_k-\bar u)|\, |\varphi|\, dx \leq \int_{\mW_k}|f(x,u_k)-f_\infty(x,u_k)|\, |\varphi|\, dx\\ 
  &+ \int_{\mW_k}|f_\infty(u_k)-f_\infty(u_k-\bar u)|\, |\varphi|\, dx 
  + \int_{\mW_k}|f_\infty(u_k-\bar u)-f(x,u_k-\bar u)|\, |\varphi|\, dx\\
  &< s_0\eps\|\varphi\|_{L^2} |\mW_k|^{\frac12} \leq s_0\eps \frac{1}{s_0}\|u_k\|_{L^2} \|\varphi\|_{L^2} \leq \eps \tilde{C}\|\varphi\|_{L^2},
\end{align*}
and
\begin{align*}
  &\int_{\mW_k}|F(x,u_k)-F(x,u_k-\bar u)|\, dx \leq \int_{\mW_k}\int_0^1|f(x,t u_k)-f_\infty(x,t u_k)|\, |u_k|\,dt dx\\ 
  &+ \int_{\mW_k} |F_\infty(u_k)-F_\infty(u_k-\bar u)|\, dx\\
  &+ \int_{\mW_k}\int_0^1 |f_\infty(t(u_k-\bar u))-f(x,t(u_k-\bar u))|\, |u_k-\bar u|\, dt dx\\
  &< \frac{s_0\eps}{3} |\mW_k|^{\frac12}\,(\|u_k\|_{L^2}+s_0|\mW_k|^{\frac12}  + \|u_k-\bar u\|_{L^2}) \leq \eps \frac{4\tilde{C}^2}{3}
\end{align*}
for all $k\in\N$. Finally, since $u_k\to \bar u$ strongly in $L^r(B_{R_0}(0))$, $2\leq r<2^\ast$, we can choose $k_0\in\N$ large enough such that
\begin{align*}
  &\int_{B_{R_0}(0)}|f(x,u_k)-f(x,\bar u)-f(x,u_k-\bar u)|\, |\varphi|\, dx \leq \eps \|\varphi\| \\
  \text{and  } &\int_{B_{R_0}(0)}|F(x,u_k)-F(x,\bar u)-F(x,u_k-\bar u)|\, dx \leq \eps
\end{align*}
hold for all $k\geq k_0$. Combining the above estimates, and remarking that 
$(\mU_k\cup \mV_s^k\cup \mW_k)\backslash B_{R_0}(0)=\R^N\backslash B_{R_0}(0)$ holds for all $k$,
we obtain
\begin{align*}
  &\int_{\R^N}|f(x,u_k)-f(x,u_k-\bar u)-f(x,\bar u)|\, |\varphi|\, dx \leq 3\eps (1+\tilde{C}) \|\varphi\| \\
  \text{and }&\int_{\R^N}|F(x,u_k)-F(x,u_k-\bar u)-F(x,\bar u)|\, dx \leq \eps (2+\frac{23}{6}\tilde{C}^2)
\end{align*}
for all $k\geq k_0$. This concludes the proof.
\end{proof}



\begin{thebibliography}{00}

\bibitem{AckWeth05} N. Ackermann and T. Weth, 
Multibump solutions of periodic Schr{\"o}dinger equations in a degenerate setting,
{\em Comm. Contemp. Math.} {\bf 7} (2005) 269--298.

\bibitem{ambrosetti-rabinowitz:73} A. Ambrosetti and P. Rabinowitz, 
Dual variational methods in critical point theory and applications.
{\em J. Functional Analysis} {\bf 14} (1973), 349--381. 

\bibitem{BaLi90} A. Bahri and Y.Y. Li, 
On a min-max procedure for the existence of a positive solution for certain scalar field equations in $\R^N$, 
{\em Rev. Mat. Iberoamericana} {\bf 6} (1990) 1--15.

\bibitem{BaLions97} A. Bahri and P.-L. Lions, 
On the existence of a positive solution of semilinear elliptic equations in unbounded domains, 
{\em Ann. Inst. H. Poincar{\'e} Anal. Non Lin{\'e}aire} {\bf 14} (1997) 365--413.

\bibitem{BeGaKa83} H. Berestycki, T. Gallou{\"e}t and O. Kavian, 
{\'E}quations de champs scalaires euclidiens non lin{\'e}aires dans le plan, 
{\em C.R. Acad. Sc. Paris, S{\'e}rie I} {\bf 297} (1983) 307--310.

\bibitem{Berest_Lions83} H. Berestycki and P.-L. Lions, 
Nonlinear scalar field equations, I. Existence of a ground state,
{\em Arch. Rational Mech. Anal.} {\bf 82} (1983) 313--345.

\bibitem{cds} G. Cerami, G. Devillanova and S. Solimini, 
Infinitely many bound states for some nonlinear scalar field equations, 
{\em Calc. Var. Partial Diff. Equations} {\bf 23} (2005) 139--168.

\bibitem{CW04} M. Clapp and T. Weth, 
Multiple solutions of nonlinear scalar field equations, 
{\em Comm. Part. Diff. Equations} {\bf 29} (2004) 1533--1554.

\bibitem{dn}  W. Y. Ding and W.-M. Ni, 
On the existence of positive entire solutions of a semilinear elliptic equation, 
{\em Arch. Rational Mech. Anal. } {\bf 91} (1986) 283--308.

\bibitem{GHOU} N. Ghoussoub, 
{\em Duality and perturbation methods on critical point theory}, 
Cambridge Tracts in Mathematics, 107, Cambridge University Press, Cambridge, 1993.

\bibitem{GNN} B. Gidas, W.M. Ni and L. Nirenberg, 
Symmetry of positive solutions of nonlinear elliptic equations in $\mathbb{R}^n$,
in: {\em Mathematical analysis and applications, Part A (ed. L. Nachbin)}, 
Adv. in Math. Suppl. Stud., 7A, Academic Press, New York--London, 1981, pp. 369--402.

\bibitem{h}  N. Hirano, 
Multiple existence of sign changing solutions for semilinear elliptic problems on $\mathbb{R}^{N},$ 
{\em Nonlinear Anal.} {\bf 46} (2001), 997--1020.

\bibitem{HuangWang08} Y. Huang and F. Wang, 
On a class of Schr{\"o}dinger equations in $\mathbb{R}^N$ with indefinite linear part,
{\em Nonlinear Anal.} {\bf 69} (2008) 4504--4513.

\bibitem{JANG10} J.D. Jang, 
Uniqueness of positive radial solutions of $\Delta u+f(u)=0$ in $\mathbb{R}^N$, $N\geq 2$,
{\em Nonlinear Anal.} {\bf 73} (2010) 2189--2198.

\bibitem{KrSz98} W. Kryszewski and A. Szulkin, 
Generalized linking theorem with an application to a semilinear Schr{\"o}dinger equation, 
{\em Adv. Differential Equations} {\bf 3} (1998) 441--472.

\bibitem{li}  Y. Y. Li, 
Existence of multiple solutions of semilinear elliptic equations in $\mathbb{R}^{N},$ 
in: {\em Variational Methods (Paris, 1988)}, Progr. Nonlinear Differential Equations Appl., 4, 
Birkh\"{a}user, Boston, 1990, pp. 133--159.

\bibitem{LIONS84_1} P.L. Lions, 
The concentration-compactness principle in the calculus of variations. The locally compact case I, 
{\em Ann. Inst. H. Poincar\'e Anal. Non Lin{\'e}aire} {\bf 1} (1984) 109--145.

\bibitem{LIONS84_2} P.L. Lions, 
The concentration-compactness principle in the calculus of variations. The locally compact case II, 
{\em Ann. Inst. H. Poincar\'e Anal. Non Lin{\'e}aire} {\bf 1} (1984) 223--283.

\bibitem{PANKOV05} A. Pankov, 
Periodic nonlinear Schr\"odinger equation with application to photonic crystals,
{\em Milan J. Math.} {\bf 73} (2005) 259--287.

\bibitem{PANKOV08} A. Pankov, 
On decay of solutions to nonlinear Schr\"odinger equations, 
{\em Proc. Amer. Math. Soc.} {\bf 136} (2008) 2565--2570.

\bibitem{RT} M. Ramos, H. Tavares, 
Solutions with multiple spike patterns for an elliptic system, 
{\em Calc. Var. Partial Diff. Equations} {\bf 31} (2008) 1--25.

\bibitem{RY} M. Ramos, J.F. Yang, 
Spike-layered solutions for an elliptic system with Neumann boundary conditions, 
{\em Trans. Amer. Math. Soc.} {\bf 357} (2005) 3265--3284.

\bibitem{SIMON82} B. Simon, 
Schr\"odinger semigroups, 
{\em Bull. Amer. Math. Soc.} {\bf 7} (1982) 447--526.

\bibitem{STUART98} C.A. Stuart, 
An introduction to elliptic equations on $\R^N$, 
in: {\em Nonlinear functional analysis and applications to differential equations (Trieste, 1997)}, 
World Sci. Publ., River Edge NJ, 1998, pp. 237--285.

\bibitem{STRUWE} M. Struwe, 
{\em Variational methods: applications to nonlinear partial differential equations and Hamiltonian systems}, 
Springer-Verlag, Berlin, 1990.

\bibitem{SW2009} A. Szulkin and T. Weth, 
Ground state solutions for some indefinite variational problems, 
{\em J. Funct. Anal.} {\bf 257} (2009) 3802--3822.

\bibitem{W} A. Szulkin and T. Weth, 
The method of Nehari manifold, 
in: {\em Handbook of nonconvex Analysis and Applications, D.Y. Gao and D. Motreanu eds}, 
International Press, Boston (2010), pp. 597--632.

\bibitem{WILLEM} M. Willem, 
{\em Minimax Theorems}, 
Progr. in Nonlinear Differential Equations Appl., 24,
Birkh\"auser, Boston, 1996.

\bibitem{zh}  X. P. Zhu, 
Multiple entire solutions of a semilinear elliptic equation, 
{\em Nonlinear Anal.} {\bf 12} (1988) 1297--1316.

\end{thebibliography}
\end{document}